\documentclass[a4paper,11pt]{amsart}
\usepackage{mathrsfs}
\usepackage[all]{xy}
\usepackage{amsmath,amssymb,amscd,bbm,amsthm,mathrsfs,dsfont}
\usepackage{graphicx}
\usepackage[enableskew,vcentermath]{youngtab}
\usepackage{ytableau}
\newtheorem{thm}{Theorem}[section]
\newtheorem{lem}{Lemma}[section]
\newtheorem{cor}{Corollary}[section]
\newtheorem{prop}{Proposition}[section]
\newtheorem{rem}{Remark}[section]

\theoremstyle{definition}
\numberwithin{equation}{section}
\newcommand{\ed}{\end{document}}
\begin{document}

\title[Monogenic  functions of several vector  variables and the Dirac complex]{The Hartogs-Bochner extension for monogenic functions of several  vector variables and the Dirac complex}

\author{Yun Shi ${}^\dag$ and  Wei Wang${}^\ddag$}

\thanks{$  \dag$ Department of Mathematics,  Zhejiang University of Science and Technology, Hangzhou 310023, China,
  Email:  shiyun@zust.edu.cn;}
\thanks{${}^\ddag$ School of Mathematical Science, Zhejiang University (Zijingang campus),  Zhejiang 310058, China, Email: wwang@zju.edu.cn;}
\thanks{The first author is  partially supported by  Nature Science Foundation of Zhejiang province (No. LY22A010013) and  National Nature Science Foundation in China (Nos. 11801508, 11971425); The second author is partially supported by National Nature Science Foundation in China  (Nos. 11971425, 12371082).}

\subjclass{30G35; 58J10; 35N10; 15A66}
\keywords{monogenic  functions   of several vector variables; several Dirac operators; the Dirac  complex; tangentially monogenic functions; the Hartogs--Bochner extension}

 \begin{abstract}
Holomorphic functions in several complex variables    are generalized to  regular functions in several quaternionic variables, and further to monogenic functions of several vector variables, which are annihilated by several Dirac operators on $k$ copies of the Euclidean space $\mathbb R^n$.  As   the Dolbeault complex in complex analysis, the Dirac complex  resolving several Dirac operators  plays the  fundamental role  to investigate monogenic functions. Although the spaces in  the Dirac complex are complicated irreducible modules of ${\rm GL}(k),$ we give a simple characterization of the first four spaces, which allows us to write down first three operators in the Dirac complex explicitly and to show this part to be an elliptic complex.  Then   the PDE method can be applied  to obtain solutions to  the non-homogeneous several Dirac equations  under the  compatibility condition, which implies the Hartogs' phenomenon of monogenic functions. Moreover, we   find the boundary version of several Dirac operators and introduce the notion of a tangentially monogenic function, corresponding to tangential Cauchy--Riemann operator   and CR functions in several complex variables, and establish the Hartogs-Bochner extension of a tangentially monogenic function on the boundary of a domain.
\end{abstract}
 \maketitle
\section{Introduction}Holomorphic functions in several complex variables  are generalized to $k$-regular functions in several quaternionic variables, and further to monogenic functions of several vector variables, which are annihilated by several Dirac operators  on $k$ copies of the Euclidean space $\mathbb R^n$.  Write  $\mathbf x_A=\left(x_{A1},\dots,x_{An}\right)$ as  the vector variable in the $A$-th copy of $\mathbb R^n,$ and $\partial_{Aj}:=\frac{\partial}{\partial{x_{Aj}}},$ where $A=0,\dots,k-1,j=1, \dots,n.$  The Dirac operator in the $A$-th copy is $\nabla_A:\Gamma\left(\mathbb R^n,\mathbb S^\pm\right)\rightarrow\Gamma\left(\mathbb R^n,\mathbb S^\mp\right)$ with
$
\nabla_A:=\sum_{j=1}^n\gamma_j\partial_{A j},
$
$A=0,\dots,k-1,$ where  $\mathbb S^\pm$ denote the two spinor modules for $n$ even, the same symbols are used for $n$ odd with the convention that $\mathbb S^+$ and $\mathbb S^-$ are isomorphic, and  $\gamma_j:\mathbb S^\pm\rightarrow\mathbb S^\mp$ are Dirac matrices.   The \emph{several Dirac operators} $\mathcal D_0:\Gamma\left(\Omega,\mathbb S^+\right)\rightarrow \Gamma\left(\Omega,\mathbb C^k\otimes\mathbb S^-\right)$ are given by $$\mathcal D_0f=\left(\begin{matrix}\nabla_0f\\\vdots\\\nabla_{k-1}f \end{matrix}\right),\quad {\rm for}\ f\in\Gamma\left(\Omega,\mathbb S^+\right),$$ where $\Omega$ is a domain in $\mathbb R^{kn}.$ $f$ is called a \emph{monogenic function} on $\Omega$ if $
\mathcal D_0f(\mathbf x)=0,
$ for any $\mathbf x\in\Omega.$ The space of all monogenic  functions on $\Omega$ is denoted by $\mathcal O(\Omega).$ The notion of a   spinor-valued monogenic function on $\mathbb R^{kn}$ is the generalizing  of a holomorphic function on $\mathbb C^n.$ Ones have been interested in generalization the theory of several complex variables to this kind of functions since 90's \cite{SV}.   For this purpose, a fundamental tool is to solve  the non-homogeneous equation
\begin{align} \label{dfg}
\mathcal D_0f=g,
\end{align}
for $g\in\Gamma\left(\Omega,\mathbb C^k\otimes \mathbb S^-\right).$ Since it   is overdetermined, (\ref{dfg}) can only be solved under a   compatibility condition. Thus the solution of this problem can be obtained if one can provide a description of the so-called the Dirac complex. In the quaternionic case, such complexes are known explicitly now (cf. e.g. \cite{MR1165872},  \cite{bures1}-\cite{CSSS},   \cite{Wa11}). They are further extended to the Heisenberg groups \cite{Ren,SW2,SW,Wa11}.

For $2$ vector variables with the dimension $n$  even, the Dirac complex was found by Krump-Sou\v cek \cite{Krump}, by using parabolic geometry associated to ${\rm SO}(2,2n+2).$ In \cite{SWW}, the authors wrote down operators explicitly in the Dirac complex of two variables, without the restriction of even dimension, and applied the PDE method of several complex variables to equation (\ref{dfg}) to obtain compactly support solution if $g$ is compactly supported and satisfies the compatibility condition\begin{align}\label{D1g}
\mathcal D_1g=0.
\end{align}  This solution implies the Hartogs' phenomenon for monogenic functions of $2$ vector variables.

For $3$ vector variables  with the dimension $n$   greater than   $5,$ Sabadini-Sommen-Struppa-Van Lancker  \cite{SSSL} wrote down operators in the Dirac complex explicitly, with the help of the tool of  megaforms and the method of computer algebra. In the stable case (i.e. $k\leq \frac{n}{2}$),  the method of parabolic geometry was used  to construct the Dirac complex by 
Krump \cite{kru} for $n=4$ and by Sala\v c \cite{Salac}-\cite{Salac18} for general  even  $n$.  Sala\v c also constructed  the second order differential operators in the complex as the sum of $4$ invariant differential  operators \cite{Salac}. For the unstable case, it is an open problem to construct the Dirac complex. However  some  results   can be found in Krump \cite{kru2}.

In this paper, we   write down explicitly the first three operators in the Dirac complex for any $n$ and $k$. For $k$ ($k\geq3$) vector variables, the first four terms of this  complex can be written as
\begin{equation}\begin{aligned}\label{co3}
0\rightarrow\Gamma\left(\Omega,\mathcal V_0 \right)\xrightarrow{\mathcal{D}_{0}} \Gamma\left(\Omega,\mathcal V_1 \right)\xrightarrow{\mathcal{D}_{1}}\Gamma\left(\Omega,\mathcal V_2\right) \xrightarrow[\mathcal{D}''_{2}]{\mathcal{D}'_{2}}\begin{matrix} \Gamma\left(\Omega,\mathcal V_3'\right)\\\oplus\\\Gamma\left(\Omega,\mathcal V_3''\right)\end{matrix},
\end{aligned}\end{equation}(cf. \cite{Damiano,Salac,Salac18-2,Salac18}   in the stable case),
where $\Omega$ is a domain in $\mathbb R^{kn},$ and
\begin{align*}\mathcal V_0=& V_{0\dots}\otimes\mathbb S^+,\\ \mathcal V_1=&V_{10\dots}\otimes\mathbb S^-,\\\mathcal V_2=& V_{210\dots}\otimes\mathbb S^-,\\\mathcal V_3'=& V_{220\dots}\otimes\mathbb S^+,\\ \mathcal V_3''=&  V_{3110\dots}\otimes\mathbb S^-.\end{align*}
Here  $V_\lambda$ is an irreducible ${\rm GL}(k)$-module labeled by $\lambda,$ where $\lambda=\left(\lambda_1,\dots,\lambda_k\right)$ is a partition of $n,$ i.e.   $\lambda_1\geq\lambda_2\geq\dots\lambda_k\geq0,$  and  $\lambda_1+\dots+\lambda_k=n.$  $V_{0\dots}\cong\mathbb C,V_{10\dots}\cong\mathbb C^k.$ The spaces $V_{210\dots},V_{220\dots},V_{3110\dots}$ can be realized as subspaces of tensor products of $\mathbb C^k$ in terms of Young symmetrizer, which are called Weyl modules. We give a very simple and elementary characterization of these spaces as follows.

 Let $\omega^A,A=0,1,\dots,k-1,$ be  a basis of $\mathbb C^k.$ Denote briefly $\omega^{A_1\dots A_l}:=\omega^{A_1}\otimes\cdots\otimes\omega^{A_l}.$ An element $h$ in $\otimes^l\mathbb C^k$ can be written as $h=h_{A_1\dots A_l}\omega^{A_1\dots A_l}.$ Here and in the sequel  we use Einstein's convention of taking summation over repeated indices. We also identify an element $h$ with the tuple $\left(h_{A_1\dots A_l}\right).$
\begin{prop}\label{p1}
{\rm (1)}
$h=\left(h_{ABC}\right) \in V_{210\dots}$ if and only if   \begin{equation}
\begin{aligned}\label{HABC}
h_{[A\underline{B}C]}+h_{[A\underline{C}B]}=\frac32h_{ABC}.
\end{aligned}\end{equation}Moreover, for any $h \in V_{210\dots},$ we have \begin{align}\label{Hcomm}
h_{ABC}=h_{ACB}.
\end{align}
{\rm (2)} $h=\left(h_{D ABC}\right)\in V_{220\dots}$ if and only if
\begin{align}\label{HABCD}
h_{D ABC}=\frac16\sum_{(A,D),(B,C)}\left(h_{D[A\underline{B}C]}+ h_{B[C\underline{D}A]}\right).
\end{align}
{\rm (3)} $h=\left(h_{EDABC}\right)\in  V_{3110\dots}$ if and only if   \begin{equation}
\begin{aligned}\label{HEDABC}
\sum_{(D,B,C)}h_{[E\underline{D}A\underline{B}C]} =\frac{10}{3}h_{EDABC}.
\end{aligned}\end{equation}Here bracket $[\cdots]$ means  skewsymmetrization of   indices, but underlined indices in a bracket are not skewsymmetrized.        $\sum_{(B,\cdots,C)}$ denotes the summation   taken over all permutations $(B,\cdots,C).$
\end{prop}It is convenient to use these characterization as the  definition of these vector  spaces.

A section in $\Gamma\left(\Omega,\mathcal V_0\right)$ is an   $\mathbb S^+$-valued function on $\Omega$, while a section in $\Gamma\left(\Omega,\mathcal V_1\right)$ is written as   $F=\left(F_A\right)$ with $F_A$ to be    $\mathbb S^-$-valued functions on $\Omega.$ A section in $\Gamma\left(\Omega,\mathcal V_2\right)$ is written as $h=\left(h_{ABC}\right)$ with $h_{ABC}$ to be  $\mathbb S^-$-valued functions  which satisfy (\ref{HABC}) on $\Omega.$ A section of $\Gamma\left(\Omega,\mathcal V_3'\right)$ or $\Gamma\left(\Omega,\mathcal V_3''\right)$ is similarly defined. The   operators $\mathcal D_0,\mathcal D_1,  \mathcal D'_2$ and $\mathcal D''_2$ in (\ref{co3}) are   \begin{equation}\begin{aligned}\label{D0}
\left({\mathcal D}_0f\right)_A=&\nabla_A f, \\\left(\mathcal D_1F\right)_{ABC}=&\sum_{(B,C)}\nabla_{[A}\nabla_{\underline{B}} F_{C]},   \\ \left(\mathcal D'_2h\right)_{D ABC}=&\sum_{(A,D),(B,C)}\left(\nabla_{D}h_{[A\underline{B}C]}+ \nabla_{B}h_{[C\underline{D}A]}\right),
 \\ \left(\mathcal D''_2h\right)_{E D ABC}=&
 \frac12\sum_{(D,B,C)}\left( 2\nabla_{[E}\nabla_{\underline{D}}h_{A]BC} +\nabla_{D}\nabla_{[E}h_{A]BC} +\Delta_{BC}h_{[E\underline{D}A]}\right), \end{aligned}\end{equation}
for $f\in\Gamma\left(\Omega,\mathcal V_0\right),F\in\Gamma\left(\Omega,\mathcal V_1\right), h\in\Gamma\left(\Omega,\mathcal V_2\right),$ where  $A,B,C,D,E=0,1,\dots,k-1.$ We can check both $\mathcal D_2'h$ and  $\mathcal D_2''h$ are $\mathcal V_3''$-valued (cf. Corollary \ref{cor21}). So we have the operator  $$\mathcal D_2=\mathcal D'_2+\mathcal D''_2:\Gamma \left(\Omega,\mathcal V_2\right)\rightarrow\Gamma \left(\Omega,\mathcal V_3\right),\qquad \mathcal V_3=\mathcal V'_3\oplus\mathcal V''_3.$$ For $3$ vector variables, the differential operators in the Dirac complex was given explicitly by \cite{Damiano,SSSL}. We can check that in this case, operators in (\ref{D0}) coincide  with their formulae.

Then we prove the short sequence   (\ref{co3}) is an elliptic complex.   The symbol sequence of {\rm(\ref{co3})} at $\xi \in\mathbb R^{kn}\setminus\{\mathbf 0\},$    is
\begin{equation}\begin{aligned}\label{simco}
0\rightarrow\mathbb S^+\xrightarrow{\sigma_{0}(\xi)} \mathcal V_1\xrightarrow{\sigma_{1}(\xi)}  \mathcal V_2 \xrightarrow[\sigma_{2}''(\xi)]{\sigma_{2}'(\xi)} \begin{matrix} \mathcal V_3'\\\oplus\\\mathcal V_3''\end{matrix},
\end{aligned}\end{equation}  where $\sigma_j(\xi)$ is the symbol of the operator $\mathcal D_j$ in (\ref{D0}).

\begin{thm}\label{exact}For $k\geq3,$ the short sequence   {\rm(\ref{co3})} is an elliptic complex, i.e.  $\mathcal D_{j+1}\circ\mathcal D_j=0$ for $j=0,1,$ and  $\ker\sigma_0{(\xi)}=\{0\},$ and   $\ker\sigma_j{(\xi)}={\rm Im}\ \sigma_{j-1}{(\xi)},  j=1,2,$  for $\xi\in\mathbb R^{kn}\setminus\{\mathbf 0\}.$
\end{thm}
One difficulty of checking (\ref{co3}) to be a complex is that the representation spaces   $V_{210\dots},V_{220\dots}$ and $V_{3110\dots}$ are not easily manipulated as exterior or symmetric powers. It becomes more tedious but proofs are still elementary, based on the characterization in Proposition \ref{HABC}. Another one    comes  from the noncommutativity  of several  Dirac operators  $\nabla_A$'s. But their anticommutators $$
\left\{\nabla_B,\nabla_C\right\}=\nabla_B\nabla_C+\nabla_C\nabla_B =:\Delta_{BC}$$ are scalar operators, which commutate with any $\nabla_A,$ i.e.
\begin{align}\label{Delta}
\Delta_{BC}\nabla_A=\nabla_A\Delta_{BC}.
\end{align}
   This nice property is frequently used in this paper.

Then as in the quaternionic case \cite{wang2008,Wa10}, \cite{wang23}-\cite{wang29}, we can apply the  PDE method of several complex variable to study this complex. We show the non-homogeneous equation (\ref{dfg}) has  compactly supported solution  if $g$ is compactly supported and satisfies the compatibility condition \eqref{D1g}. This result implies the Hartogs' phenomenon for monogenic functions of $k$ vector variables of any dimensions.

To investigate monogenic functions on a domain, we need  the boundary  complex of the Dirac complex (\ref{co3}). By using the construction  of the boundary complex of a general differential complex  (see eg. \cite{Andreotti,AN,Naci,Nacinovich85}),  we get the tangential several Dirac operators:
  \begin{align*}
 {\mathscr D}_0:\Gamma\left(\partial \Omega,\mathcal V_0\right)&\rightarrow \Gamma\left(\partial \Omega,\mathscr V_1\right),
  \end{align*}
where $\mathscr V_1$ is two copies of $\mathbb S^-\otimes\mathbb C^{k-1},$ i.e.  \begin{align*}
  \mathscr V_1\cong\left(\mathbb S^-\otimes\mathbb C^{k-1}\right)\oplus\left(\mathbb S^-\otimes\mathbb C^{k-1}\right).
  \end{align*}
It is  the counterpart of tangential Cauchy-Riemann operator on the boundary of a domain in $\mathbb C^n.$   $\hat f\in \Gamma\left(\partial \Omega,\mathcal V_0\right)$ is called \emph{tangentially monogenic} if \begin{align*} {\mathscr D}_0\hat f=0.
\end{align*}
  This is the counterpart of the notion of a CR function  in several complex variables. Then we can establish  the Hartogs-Bochner extension  for tangentially monogenic functions.
\begin{thm}\label{HBE} Suppose that $k\geq2$ and  $\Omega$ is a bounded domain in $\mathbb R^{kn}$ with smooth boundary such that $\mathbb R^{kn}\setminus \overline \Omega$ is connected. If $f$ is a smooth tangentially monogenic function on $\partial \Omega$, then there exists $\tilde f\in\mathcal O(\Omega)$ smooth up to the boundary such that $\tilde f=f$ on $\partial \Omega.$
\end{thm}

In the quaternionic case, the boundary complex of $k$-Cauchy-Fueter complex was obtained by the second author in \cite{wang29}, and used to establish the Hartogs-Bochner extension for $k$-CF functions on the boundary.

The paper is organized as follows. In Section 2,  the characterization of vector spaces $V_{210\dots},$ $V_{220\dots},$ $V_{3110\dots}$ in the Dirac complex is given. In Section 3, we show the short sequence (\ref{co3}) with operators given by (\ref{D0}) is a complex.  In Section 4,  we study  vector spaces of the boundary complex of the  Dirac complex and obtain  the boundary version of several Dirac operators. In Section 5,   we prove the ellipticity of the short sequence (\ref{co3}) of the  Dirac complex and    the Hartogs-Bochner extension  for tangentially  monogenic functions of $k$  vector variables by using  uniform elliptic operators,   the associated Hodge-Laplacian operators, to solve the non-homogeneous equation (\ref{dfg}) under the compatibility condition (\ref{D1g}). In this approach, operators $\mathcal D_2'$ and $\mathcal D_2''$ are used to solve (\ref{dfg}).   In Appendix \ref{app}, we give some basic facts  about Young diagrams, Young symmetrizer and the construction of Weyl modules as  irreducible representations of ${\rm GL}(k).$

\section{Characterization of vector spaces  in the short sequence of  the Dirac complex}
\subsection{Spinor modules and   ${\rm GL}(k)$-modules}The {\it real Clifford algebra} $\mathbb R_n$ is the associative algebra generated by  the $n$ basis elements of $\mathbb R^n$ satisfying
\begin{align*}
e_ie_j+e_je_i=-2\delta_{ij},
\end{align*}
for $i,j=1,\dots,n.$ The basis of $\mathbb R_n$ is $e_0=1,\  e_\alpha=e_{a_1}\dots e_{a_i},$ for $1\leq i\leq n,1\leq a_1<\dots<a_i\leq n.$

When $n$ is even, there are two spinor modules $\mathbb S^\pm.$ Let $\gamma_j^\pm:\mathbb S^\pm\rightarrow\mathbb S^\mp$
be the representation of $e_j$ on the spinor modules. In odd dimensional $2m+1,$ there exists a unique spinor module $\mathbb S=\mathbb S^+\oplus\mathbb S^-,$ where $\mathbb S^\pm$ are spinor modules of  $\mathfrak{so}(2m,\mathbb C),$   and  $\gamma_j=\left(\begin{matrix}0&\gamma_j^+\\\gamma_j^-&0 \end{matrix}\right)$ (cf. \cite{Delanghe}).  $\gamma_j^\pm$ will be denoted as $\gamma_j$ for simplicity. They satisfy
\begin{align}\label{comm}
\gamma_j\gamma_k+\gamma_k\gamma_j=-2\delta_{jk} {\rm Id}_{\mathbb S^\pm}.
\end{align}
In the sequel  of  this paper,  $\mathbb S^\pm$  denote the two spinor modules for $n$ even, and the same symbols are used for $n$ odd with the convention that $\mathbb S^+$ and $\mathbb S^-$ are isomorphic.

\begin{lem}\label{gamma} { \rm(cf. e.g.  \cite[p.125]{Delanghe} or  \cite[Lemma 2.1]{SWW})} The adjoint $\gamma_k^*$   of $\gamma_k$ with respect to the standard   inner product on $\mathbb S^\pm$ is given by
\begin{align}\label{star}\gamma_k^*=-  \gamma_k.\end{align}
\end{lem}

\emph{Young diagram} associated to a partition $\lambda=\left(\lambda_1,\dots,\lambda_k\right)$ is a
 \begin{equation*}\begin{aligned}
&\young(~~~)\dots \dots\young(~)\\[-1.9mm]&\young(~~~)\dots
\\[-1.9mm]&\vdots\quad\vdots\\[-1.9mm]&\young(~)\dots
\end{aligned}\end{equation*}
with $\lambda_i$ boxes in the $i$-th row. For $k=2,$ the Dirac complex in terms of the Young diagrams  is \begin{equation}\begin{aligned}\label{k=2}
0\rightarrow\bullet\longrightarrow\young(~) \Rightarrow\young(~~,~) \longrightarrow\bullet\rightarrow0,
\end{aligned}\end{equation} which was given in \cite{Damiano,Krump}, where $\bullet=\mathbb C$. We constructed  the differential operators explicitly in   \cite{SWW} without the dimension restriction.

For $k\geq3$, the Dirac complex in terms of the Young diagrams  is
\begin{equation}\begin{aligned}\label{k=3}
0\rightarrow\bullet\longrightarrow\young(~) \Rightarrow&\young(~~,~)\quad\longrightarrow\young(~~,~~) \\&\Downarrow \qquad\qquad \quad \Downarrow\\&\young(~~~,~,~)\longrightarrow\young(~~~,~~,~) \quad \Rightarrow\young(~~~,~~~,~~)\rightarrow\dots \\&\Downarrow\qquad\qquad\quad\Downarrow\qquad\qquad\quad\Downarrow
\\&{\small\young(~~~~,~,~,~)}\rightarrow{\small\young(~~~~,~~,~,~)} \Rightarrow{\small\young(~~~~,~~~,~~,~)}\rightarrow\dots
\\&  \qquad\qquad\qquad\  \qquad\qquad\quad\ \vdots
\end{aligned}\end{equation}
(cf. \cite{Damiano,Salac,Salac18-2,Salac18}   in the stable case), where $\rightarrow$ and $\Rightarrow$ are  differential operators of first order and  second order, respectively,  and $$\mathcal V_3=\mathcal V_3'\oplus\mathcal V_3''=\young(~~,~~)\otimes\mathbb S^+\oplus\young(~~~,~,~)\otimes\mathbb S^-.$$ Note that for $k=2,$ $\young(~~,~~)=\bullet$ as a ${\rm GL}(2)$-module in (\ref{k=2}) and $\young(~~~,~,~)$ is trivial.

Since $\mathcal D_j:\Gamma\left(\Omega,\mathcal V_j\right)\rightarrow \Gamma\left(\Omega,\mathcal V_{j+1}\right)$ is a $\mathfrak{gl}(k)\times\mathfrak{so}(n)$-invariant  (or ${\rm GL}(k)\times{\rm Spin}(n)$-invariant) operator,  its symbol is the projection of  $\mathfrak{gl}(k)\times\mathfrak{so}(n)$-modules \begin{align}\label{dec}
\pi_j:\mathcal V_j\otimes\left(\mathbb R^k\otimes \mathbb R^n\right)\otimes\left(\mathbb R^k\otimes \mathbb R^n\right)\rightarrow \mathcal V_{j+1}.
\end{align}
But for  $\mathcal V_1=\young(~)\otimes\mathbb S^-,$ where $\young(~)=\mathbb R^k,$ we have the decomposition
 \begin{equation*}\begin{aligned}
\young(~)\otimes\young(~)\otimes\young(~)=\left(\young(~,~)\oplus \young(~~)\right)\otimes\young(~)= \left(\young(~~,~)\oplus\young(~,~,~)\right)\oplus\left(\young(~~~) \oplus\young(~~,~)\right),
\end{aligned}\end{equation*}
by using Pieri's rule \cite[p. 225]{FH}, where $\young(~~,~)
$ appears twice. So it is difficult to determine the projection $\young(~~,~)\oplus\young(~~,~)\rightarrow\young(~~,~)$ in (\ref{dec}), which will give us the invariant  operator $\mathcal D_2'.$ While for $\mathcal V_2= \young(~~,~)\otimes\mathbb S^-,$ we have the decomposition
 \begin{equation}\begin{aligned}\label{y21}
 \young(~~,~)\otimes\young(~)\otimes\young(~)=&\left( \young(~~~,~)\oplus\young(~~,~~)\oplus\young(~~,~,~)\right) \otimes\young(~)\\=&\left(\young(~~~~,~)\oplus \young(~~~,~~)\oplus\young(~~~,~,~) \right)\oplus\left(\young(~~~,~~)\oplus \young(~~,~~,~) \right)\\&\oplus\left(\young(~~~,~,~)\oplus \young(~~,~~,~)\oplus\young(~~,~,~,~) \right),
\end{aligned}\end{equation} by using Pieri's rule, where $\young(~~~,~,~)$ appears twice again and we have the same difficulty to determine $\mathcal D_2''$.  When $n$ is even and in the stable case,  the Dirac complex is known \cite{Salac}-\cite{Salac18}. In this case, Sala\v c also constructed the projection $\pi_j$ in (\ref{dec}) and operators $\mathcal D_j$ by using the representation theory in a complicated way. But it is difficult to deduce  expressions (\ref{D0}) of $\mathcal D_j$'s   from his formulae even in the even dimensional case. Our operators $\mathcal D_1,\mathcal D_2',\mathcal D_2''$ in (\ref{D0})  have no  restriction of even dimensions.

 \subsection{Projection operators $C_\lambda$ and vector spaces $V_{210\dots},V_{220\dots}$ and $V_{3110\dots}$}
The    {\it symmetrisation}  and the  {\it skewsymmetrisation} of indices are given by
\begin{equation}\begin{aligned}\label{sym}
f_{\dots(A_1\dots A_p)\dots}:=&\frac{1}{p!}\sum_{\sigma\in S_p}f_{\dots A_{\sigma(1)}\dots A_{\sigma(p)}\dots},\\
f_{\dots[A_1\cdots A_p]\dots}:=&\frac{1}{p!}\sum_{\sigma\in S_p}{\rm sign}(\sigma)f_{\dots A_{\sigma(1)}\cdots A_{\sigma(p)}\dots},
\end{aligned}\end{equation}respectively, where the summations are taken over   the group $S_p$ of all permutations  of $p$ letters.

${\rm GL}(k)$-module $V_{210\dots},V_{220\dots},V_{3110\dots}$ can  be realized as a subspace of $\otimes^\bullet\mathbb C^k$ by Weyl's construction (See Appendix \ref{app}).
Let  $C_{21}:\otimes^3\mathbb C^k\rightarrow\otimes^3\mathbb C^k$ be a linear mapping given  by  \begin{equation}\begin{aligned}\label{H'}
C_{21}\left(h  \right)_{ABC} =&\frac23\sum_{(B,C)} h_{[A\underline{B}C]},
\end{aligned}\end{equation} for    $h\in\otimes^3\mathbb C^k.$ The space $V_{210\dots}$ is defined to be the image of $C_{21}$ in $\otimes^3\mathbb C^k.$
Let  $C_{22}:\otimes^4\mathbb C^k\rightarrow \otimes^4\mathbb C^k
$ be a linear mapping given by  { \begin{equation}\begin{aligned}\label{HH'}
C_{22}\left(h \right)_{DABC}=&\frac16\sum_{(A,D),(B,C)} \left(h_{D[A\underline{B}C]}+ h_{B[C\underline{D}A]}\right),
\end{aligned}\end{equation}}for   $h\in\otimes^4\mathbb C^k.$ The space $V_{220\dots}$ is defined to be the image of $C_{22}$ in $\otimes^4\mathbb C^k.$
Let  $C_{311}:\otimes^5\mathbb C^k\rightarrow\otimes^5\mathbb C^k $ be a linear mapping given by     {\begin{equation}\begin{aligned}\label{c311'}
 C_{311}\left(h\right)_{EDABC}
=&\frac{3}{10}\sum_{(D,B,C)} h_{[E\underline{D}A\underline{B}C]},
\end{aligned}\end{equation}}for    $h\in\otimes^5\mathbb C^k.$ The space $V_{3110\dots}$ is defined to be the image of $C_{311}$ in $\otimes^5\mathbb C^k.$
These images  are usually called \emph{Weyl modules} of $V_{210\dots},V_{220\dots}$ and $V_{3110\dots}$, respectively. More information about these representations of ${\rm GL}(k)$ is given in Appendix \ref{app}.
\begin{rem}{\rm (1)} For $h\in V_{210\dots},$ $h_{ABC}$ is symmetric in $B$ and $C$, but not skewsymmetric in $A$ and $C.$\\
{\rm (2)} For    $h\in V_{3110\dots},$ $h_{EDABC}$ is symmetric in $D,B,C$ and skewsymmetric in $E$ and $A,$ but not   in $E,A,C$.
\end{rem}
 \begin{prop}\label{pro} $C_{21},C_{22}$ and $C_{311}$ are all projection operators, i.e. $C_{\lambda}^2=C_{\lambda}$ for $\lambda=21,22,311.$
 \end{prop}
\begin{proof}
(1)    Denote $H_{ABC}:=C_{21}\left(h \right)_{ABC} =\frac23\left(h_{[A\underline{B}C]}+h_{[A\underline{C}B]} \right).$
Then  \begin{equation}\label{21'}
\begin{aligned}
C_{21}^2(h)_{ABC}=&C_{21}\left(H \right)_{ABC} = \frac13\left(H_{ABC}+H_{ACB}-H_{CBA}-H_{BCA}\right) \\=& \frac29\sum_{(B,C)}\left(h_{[A\underline{B}C]}+h_{[A\underline{C}B]} -h_{[ C\underline{B}A]}-h_{[C\underline{A}B]} \right) \\=&\frac23\sum_{(B,C)} h_{[A\underline{B}C]}   =  H_{ABC}=C_{21} (h)_{ABC},
\end{aligned}\end{equation}by definition (\ref{H'}) of $C_{21}$, where $\sum_{(B,C)}h_{[C\underline{A}B]}=0$. Thus $C_{21}$ is a projection.

(2)
Denote $H_{DABC}:=C_{22}\left(h \right)_{DABC}=\frac16\sum_{(A,D),(B,C)} \left(h_{D[A\underline{B}C]}+ h_{B[C\underline{D}A]}\right).$ Then \begin{equation}
\begin{aligned}\label{hc22}
C_{22}^2\left(h\right)_{D ABC}=C_{22}\left(H \right)_{D ABC}=\frac16\sum_{(A,D),(B,C)}\left(H_{D[A\underline{B}C]}+ H_{B[C\underline{D}A]}\right).\end{aligned}\end{equation}
Note that \begin{equation}\begin{aligned}\label{c22id0}
\sum_{(B,C)}H_{D[A\underline{B}C]}=\frac12\sum_{(B,C)}\left(H_{DABC} -H_{DCBA}\right)=H_{DABC}-\frac12\sum_{(B,C)}H_{DCBA},
\end{aligned}\end{equation}
by $H_{DABC}$  symmetric in $B,C$ by the definition. But
\begin{equation}\begin{aligned}\label{c22id2}
\sum_{(B,C)}H_{DCBA}=&\frac16\sum_{(B,C)}\sum_{(C,D)}\sum_{(A,B)} \left(h_{D[C\underline{B}A]}+ h_{B[A\underline{D}C]}\right) \\=&\frac16\sum_{(B,C)}\bigg(h_{D[C\underline{B}A]} +h_{D[C\underline{A}B]}+h_{B[A\underline{D}C]} +h_{A[B\underline{D}C]}\\&\qquad\qquad+h_{C[D\underline{B}A]} +h_{C[D\underline{A}B]}+h_{B[A\underline{C}D]} +h_{A[B\underline{C}D]}\bigg) \\=&\frac16\sum_{(B,C)}\bigg(h_{D[C\underline{B}A]} +h_{B[A\underline{D}C]}+h_{C[D\underline{A}B]} +h_{A[B\underline{C}D]} \bigg)\\=&-\frac16\sum_{(A,D),(B,C)} \left(h_{D[A\underline{B}C]}+ h_{B[C\underline{D}A]}\right)=-H_{DABC},
\end{aligned}\end{equation} where the third identity holds since we obviously have
\begin{equation}\begin{aligned}\sum_{(B,C)}h_{D[C\underline{A}B]} =\sum_{(B,C)}h_{A[B\underline{D}C]}=0,\quad
\sum_{(B,C)}\left(h_{C[D\underline{B}A]} +h_{B[A\underline{C}D]}\right)=0.
\end{aligned}\end{equation}Substitute  (\ref{c22id2}) into (\ref{c22id0}) to get
\begin{align}\label{c22id3}
\sum_{(B,C)}H_{D[A\underline{B}C]}=\frac32H_{DABC}.
\end{align}By relabeling indices, we get $
\sum_{(A,D)}H_{B[C\underline{D}A]}=\frac32H_{BCDA}.
$
Substitute  them into (\ref{hc22}) to get   \begin{equation*}
\begin{aligned}
C_{22}^2\left(h\right)_{D ABC}=&\frac14 \left(\sum_{(A,D)}H_{DABC}+\sum_{(B,C)}H_{BCDA}\right) \\=&\frac{1}{12}\sum_{(A,D),(B,C)} \bigg(h_{D[A\underline{B}C]}+ h_{B[C\underline{D}A]} + h_{B[C\underline{D}A]}+ h_{D[A\underline{B}C]}\bigg)\\=&\frac{1}{6}\sum_{(A,D),(B,C)} \left(h_{D[A\underline{B}C]}+ h_{B[C\underline{D}A]}\right)   \\=&H_{DABC}= C_{22}\left(h\right)_{D ABC}.
\end{aligned}\end{equation*}    Thus $C_{22}$ is a projection.

(3) Denote  $H_{EDABC}:=C_{311}\left(h \right)_{EDABC}=\frac{3}{10}\sum_{(D,B,C)} h_{[E\underline{D}A\underline{B}C]}. $ Then  {\small \begin{equation*}\begin{aligned}
C_{311}^2\left(h\right)_{EDABC}= & C_{311}\left(H\right)_{EDABC} = \frac{3}{10}\sum_{(D,B,C)} H_{[E\underline{D}A\underline{B}C]}
\\=&\frac{1}{20}\sum_{(D,B,C)}\left(H_{EDABC} -H_{EDCBA}-H_{ADEBC} +H_{ADCBE} + H_{CDEBA}- H_{CDABE}  \right) \\=&:S_1+S_2+S_3+S_4+S_5+S_6.
\end{aligned}\end{equation*}}Since $H_{EDABC}$ is symmetric in $D,B,C$ and skewsymmetric in $E,A$ by definition, we have \begin{equation}\begin{aligned} \label{s1}
S_1=S_3=\frac{1}{20}\sum_{(D,B,C)}H_{EDABC} =\frac{3}{10} H_{EDABC}.
\end{aligned}\end{equation}The second sum is
\begin{equation}\begin{aligned}\label{s2}
S_2=&-\frac{1}{20}\sum_{(D,B,C)}H_{EDCBA} = -\frac{3}{200}\sum_{(D,B,C)}\sum_{(D,B,A)}  h_{[E\underline{D}C\underline{B}A]} \\=&\frac{3}{200}\sum_{(D,B,C)}\sum_{(D,B,A)}  h_{[E\underline{D}A\underline{B}C]} \\=&\frac{3}{200}\sum_{(D,B,C)} \left(\sum_{(D,B)} h_{[E\underline{D}A\underline{B}C]} +\sum_{(A,B)}h_{[E\underline{A}D\underline{B}C]} +\sum_{(A,D)}h_{[E\underline{D}B\underline{A}C]}  \right) \\=&\frac{3}{200}\sum_{(D,B,C)}  \sum_{(D,B)} h_{[E\underline{D}A\underline{B}C]} =
 \frac{3}{100}\sum_{(D,B,C)}  h_{[E\underline{D}A\underline{B}C]}=\frac{1}{10}H_{EDABC}.
\end{aligned}\end{equation}Here the fifth identity holds because  $h_{[E\underline{A}D\underline{B}C]}$ is skewsymmetric in $D,C,$ which are also symmetrized,   while $h_{[E\underline{D}B\underline{A}C]}$ is skewsymmetric in $B,C,$ which are also symmetrized. By exchanging  $A$ and $E$ in (\ref{s2}), we get \begin{equation}\begin{aligned}\label{s4} S_4=\frac{1}{20}\sum_{(D,B,C)}H_{ADCBE}= -\frac{1}{10}H_{ADEBC}=\frac{1}{10}H_{EDABC}, \end{aligned}\end{equation}by  $H_{EDABC}$ skewsymmetric in $E$ and $A.$  Moreover, we have  \begin{equation}\begin{aligned}\label{s5} S_5=\frac{1}{20}\sum_{(D,B,C)}H_{CDEBA}= -\frac{1}{20}\sum_{(D,B,C)}H_{EDCBA}=S_2=\frac{1}{10}H_{EDABC},\\ S_6=-\frac{1}{20}\sum_{(D,B,C)}H_{CDABE} =\frac{1}{20}\sum_{(D,B,C)}H_{ADCBE} =S_4 =\frac{1}{10}H_{EDABC}.\end{aligned}\end{equation}  The summation of (\ref{s1})-(\ref{s5}) give us
\begin{equation}\begin{aligned}\label{311'}
C_{311}^2\left(h\right)_{EDABC} =H_{EDABC}=C_{311}\left(h\right)_{EDABC}.
\end{aligned}\end{equation}
Thus $C_{311}$ is a projection.
\end{proof}
 {\it Proof of Proposition \ref{p1}.}
(1)   Suppose that $h$ satisfy (\ref{HABC}), then $h=C_{21}(h).$ Thus $h$ belongs to the Weyl module. If $h$ belongs to the  Weyl module, then $h=C_{21}(h')$ for some $h'\in\otimes^3\mathbb C^k.$ Then $C_{21}(h)=C_{21}^2(h')=C_{21}(h')=h$ by Proposition  \ref{pro}, i.e. (\ref{HABC}) holds.

(2)-(3) can be proved similarly.
\qed

The construction of    differential operators $\mathcal D_0,\mathcal D_1,\mathcal D_2',\mathcal D_2''$ given in (\ref{co3}) is motivated by the projection maps $C_{21},C_{22},C_{311}$. Define \begin{align}\label{d2}
{\mathcal D}_2''=2\breve{\mathcal D}_2''+\widehat{\mathcal D}_2'':\Gamma \left(\Omega,\mathcal V_2\right)\rightarrow\Gamma \left(\Omega,\mathcal V_3''\right),
\end{align}
where \begin{align}\label{D''}
\breve{\mathcal D}_2''h= \sum_{(D,B,C)}\nabla_{[E}\nabla_{\underline{D}}h_{A\underline{B}C]} \omega^{EDABC},\quad \widehat{\mathcal D}_2''h= \sum_{(D,B,C)}\nabla_{D}\nabla_{[E}h_{A\underline{B}C]} \omega^{EDABC}.
\end{align}
In fact, by comparing (\ref{D0}) and \eqref{D''} with (\ref{H'})-(\ref{c311'}), we have
 \begin{equation}\begin{aligned}\label{D1}
\mathcal D_1F=& \frac32 C_{21}\left(\nabla_A\nabla_BF_C\omega^{ABC}\right),\\  
\mathcal D_2'h=&6 C_{22}\left(\nabla_Dh_{ABC}\omega^{DABC}\right),\\\breve{\mathcal D}_2''h=&\frac{10}{3}C_{311}\left(\nabla_E\nabla_Dh_{ABC} \omega^{EDABC}\right),\\\widehat{\mathcal D}_2''h=&\frac{10}{3}C_{311}\left(\nabla_D\nabla_E h_{ABC} \omega^{EDABC}\right).
\end{aligned}\end{equation}Because the image of $C_\lambda$ is $ V_\lambda,$ we have the following corollary.
\begin{cor}\label{cor21} (1)
$\mathcal D_0f\in\Gamma\left(\mathcal V_1\right)$ for $f\in\Gamma\left(\mathcal V_0\right);$ (2) $\mathcal D_1F\in\Gamma\left(\mathcal V_2\right)$ for $F\in\Gamma\left(\mathcal V_1\right);$  (3) $\mathcal D_2'h\in\Gamma\left(\mathcal V_3'\right) $  and $\mathcal D_2''h\in\Gamma\left(\mathcal V_3''\right)$ for $h\in\Gamma\left(\mathcal V_2\right).$
\end{cor}
\section{Differential operators in the short sequence of the  Dirac complex}
\subsection{The   explicit expression of $\mathcal D_2''$} \begin{prop}For $h\in\Gamma\left(\mathcal V_2\right),$
\begin{align}\label{D2}
\left({\mathcal D}_2''h\right)_{EDABC}=\frac12\sum_{(D,B,C)}\left( 2\nabla_{[E}\nabla_{\underline{D}}h_{A]BC} +\nabla_{D}\nabla_{[E}h_{A]BC} +\Delta_{BC}h_{[E\underline{D}A]}\right).
\end{align}
\end{prop}
\begin{proof}
Note that
{\small\begin{equation*}\begin{aligned}
6\sum_{(D,B,C)} \nabla_{[E}\nabla_{\underline{D}}h_{A\underline{B}C]}
=&2\sum_{(D,B,C)}\left(\nabla_{E}\nabla_{{D}}h_{[A\underline{B}C]} -\nabla_{A}\nabla_{{D}}h_{[E\underline{B}C]} +\nabla_{C}\nabla_{{D}}h_{[E\underline{B}A]}\right) \\=&\sum_{(D,B,C)}\left(\nabla_{E}\nabla_{{D}} \sum_{(B,C)}h_{[A\underline{B}C]} -\nabla_{A}\nabla_{{D}}\sum_{(B,C)}h_{[E\underline{B}C]} +\sum_{(B,C)}\nabla_{B}\nabla_{C}h_{[E\underline{D}A]}\right)
\\=&\sum_{(D,B,C)}\left(\frac32 \nabla_{E}\nabla_{D}h_{ABC} - \frac32\nabla_{A}\nabla_{D}h_{EBC}  +\Delta_{BC}h_{[E\underline{D}A]} \right)\\=& \sum_{(D,B,C)} \left(3\nabla_{[E}\nabla_{\underline{D}}h_{A]BC} +\Delta_{BC}h_{[E\underline{D}A]}\right),
\end{aligned}\end{equation*}}
by using characterization  (\ref{HABC}) of $h\in\Gamma\left(\mathcal V_2\right)$ and \begin{equation}\begin{aligned}\label{symsym}
\sum_{(D,B,C)}f_{\dots D\dots B \dots C\dots}=&\frac12\sum_{(D,B,C)}\sum_{(B,C)}f_{\dots D\dots B \dots C\dots}=\sum_{(D,B,C)} f_{\dots D\dots C \dots B\dots},\\ f_{\dots[EAC]\dots}=&\frac13\left(f_{\dots E[AC]\dots}-f_{\dots A[EC]\dots}+f_{\dots C[EA]\dots}\right),
\end{aligned}\end{equation} by the definitions of symmetrization and  skewsymmetrization in (\ref{sym}). Similarly,
{\small\begin{equation*}\begin{aligned} 6\sum_{(D,B,C)}\nabla_{D}\nabla_{[E}h_{A\underline{B}C]} =&2\sum_{(D,B,C)}\left(\nabla_{D}\nabla_{E}h_{[A\underline{B}C]} -\nabla_{D}\nabla_{A}h_{[E\underline{B}C]} +\nabla_{D}\nabla_{C}h_{[E\underline{B}A]}  \right) \\=&\sum_{(D,B,C)}\left(\nabla_{D}\nabla_{E} \sum_{(B,C)}h_{[A\underline{B}C]} -\nabla_{D}\nabla_{A}\sum_{(B,C)}h_{[E\underline{B}C]} +\sum_{(B,C)}\nabla_{B}\nabla_{C}h_{[E\underline{D}A]}  \right) \\=&\sum_{(D,B,C)}\left( \frac32 \nabla_{D}\nabla_{E}h_{ABC} - \frac32\nabla_{D}\nabla_{A}h_{EBC} +\Delta_{BC}h_{[E\underline{D}A]}\right)  \\=&\sum_{(D,B,C)} \left(3\nabla_{D}\nabla_{[E}h_{A]BC} +\Delta_{BC}h_{[E\underline{D}A]}\right).
\end{aligned}\end{equation*}}Then (\ref{D2}) follows form the summation of the above two identities  and  the definition of $\mathcal D_2''$  in (\ref{d2})-(\ref{D''}).
\end{proof}
\subsection{The short sequence (\ref{co3}) is a complex}
\begin{thm}\label{pell}
{\rm (1)} $\mathcal D_1\circ\mathcal D_0=0.$ {\rm (2)} $\mathcal D_2'\circ\mathcal D_1=0.$  {\rm (3)} $\mathcal D_2''\circ\mathcal D_1=0.$
\end{thm}
\begin{proof}(1) Note that {\small\begin{equation}\label{dF}
\left(\mathcal D_1F\right)_{ABC}=\frac12\sum_{(B,C)}\left(\nabla_A\nabla_{B}F_{C} -\nabla_C\nabla_BF_A\right)=\nabla_A\nabla_{(B}F_{C)} -\frac12\Delta_{BC}F_A,
\end{equation}}for $F\in\Gamma\left(\Omega,\mathcal V_1\right)$ by definition (\ref{D0}), where \begin{equation*}\begin{aligned}\Delta_{AB}:=& \nabla_A\nabla_B +\nabla_B\nabla_A\\=& \sum_{j,k=1}^n \gamma_j\partial_{Aj} \gamma_k\partial_{Bk} + \sum_{j,k=1}^n\gamma_k\partial_{Bk}\gamma_j\partial_{Aj} =-2\sum_{j=1}^n\partial_{Aj} \partial_{Bj}\ {\rm Id}_{\mathbb S\pm}\end{aligned}\end{equation*}  is symmetric in $A$ and $B.$ Given  $f\in\Gamma\left(\Omega,\mathcal V_0\right),$ and $A,B,C=0,\dots,k-1,$ we have
\begin{equation*} \begin{aligned}\left(\mathcal D_1\circ \mathcal D_0f\right)_{ABC}= \nabla_{A}\nabla_{(B} \nabla_{C)}f-\frac12\Delta_{BC}\nabla_{A}f =\frac12\left(\nabla_A\Delta_{BC}-\Delta_{BC}\nabla_A\right)f=0
\end{aligned}\end{equation*}by (\ref{Delta}).\\
(2) Given  $F\in\Gamma\left(\Omega,\mathcal V_1\right) $  and $D,A,B,C=0,\dots,k-1,$ we have
{ \begin{equation}\begin{aligned}\label{d2d1}
\left(\mathcal D_2'\circ\mathcal D_1F\right)_{D ABC}=&
\sum_{(A,D),(B,C)}\left(\nabla_{D}\left(\mathcal D_1F\right)_{[A\underline{B}C]}+ \nabla_{B}\left(\mathcal D_1F\right)_{[C\underline{D}A]}\right)
\\=&\sum_{(A,D)} \nabla_{D}\sum_{(B,C)}\left(\mathcal D_1F\right)_{[A\underline{B}C]} +\sum_{(B,C)} \nabla_{B}\sum_{(A,D)}\left(\mathcal D_1F\right)_{[C\underline{D}A]}
\\=&\frac32\sum_{(A,D)}\nabla_{D}\left(\mathcal D_1F\right)_{A{B}C}+\frac32\sum_{(B,C)} \nabla_{B}\left(\mathcal D_1F\right)_{C{D}A},
\end{aligned}\end{equation}}
by using  (\ref{HABC}), since $\mathcal D_1F\in \mathcal V_2$ by Corollary \ref{cor21}.   But
\begin{equation*}\begin{aligned}
\sum_{(A,D)}\nabla_D\left(\mathcal D_1F\right)_{ABC}=& \sum_{(A,D)}\nabla_D \left(\nabla_A\nabla_{(B}F_{C)} -\frac12\Delta_{BC}F_A\right)  =  \Delta_{AD}\nabla_{(B}F_{C)} -\Delta_{BC}\nabla_{(D}F_{A)},
\\\sum_{(B,C)}\nabla_{B}\left(\mathcal D_1F\right)_{C DA} =&\sum_{(B,C)}\nabla_B\left(\nabla_{C}\nabla_{(D} F_{A)}-\frac12\Delta_{AD}F_{C}\right)  =  \Delta_{BC}\nabla_{(D}F_{A)} -\Delta_{AD}\nabla_{(B}F_{C)},
\end{aligned}\end{equation*}
by using (\ref{dF}) and (\ref{Delta}). Their summation vanishes, i.e.
$\left(\mathcal D_2'\circ\mathcal D_1F\right)_{D ABC}=0.$\\
(3) By the formulae of $\mathcal D_2''$ in (\ref{D2}), we have
{\small\begin{equation}\begin{aligned}\label{T123}
4\left({\mathcal D}_2''\right.\circ&\left.\mathcal D_1F\right)_{EDABC}=2\sum_{(D,B,C)}\left(2\nabla_{[E} \nabla_{\underline{D}}\left(\mathcal D_1F\right)_{A]BC} +\nabla_{D}\nabla_{[E}\left(\mathcal D_1F\right)_{A]BC} +\Delta_{BC}\left(\mathcal D_1F\right)_{[E\underline{D}A]}\right)
\\=&\sum_{(D,B,C)}\bigg( \left(\nabla_E\nabla_D+\Delta_{DE}\right)\left(\mathcal D_1F\right)_{ABC}-\left( \nabla_A\nabla_D+\Delta_{DA}\right)\left(\mathcal D_1F\right)_{EBC}   +2\Delta_{BC}\left(\mathcal D_1F\right)_{[E\underline{D}A]}\bigg),
\end{aligned}\end{equation}}by
\begin{equation*}
   2\nabla_E\nabla_D+\nabla_D\nabla_E=\nabla_E\nabla_D+\Delta_{DE}.
\end{equation*}
Note that{\small\begin{equation}\begin{aligned}\label{dFF}
2\Delta_{BC}\left(\mathcal D_1F\right)_{[E\underline{D}A]}=& \Delta_{BC}\left(\left(\mathcal D_1F\right)_{E {D}A}- \left(\mathcal D_1F\right)_{A{D}E}\right) \\=&\frac12\Delta_{BC}\bigg(\left(\nabla_E\nabla_D +\Delta_{DE}\right)F_A -\left(\nabla_A\nabla_D+\Delta_{DA}\right)F_E +\left[\nabla_E,\nabla_A\right]F_D\bigg),
\end{aligned}\end{equation}}
by (\ref{dF}) again, and
{ \begin{equation}\begin{aligned}\label{1} \left(\nabla_{E}\nabla_{D} +\Delta_{DE}\right)\left(\mathcal D_1F\right)_{ABC} =&\frac12\left(\nabla_{E}\nabla_{D} +\Delta_{DE}\right)\left(-\Delta_{BC}\cdot F_A +2\nabla_A\nabla_{(B}F_{C)}\right),
\end{aligned}\end{equation}}by    (\ref{dF}).
By exchanging  $A$ and  $E$ in   (\ref{1}), we get
{ \begin{equation}\begin{aligned}\label{3} -\left(\nabla_{A}\nabla_{D} +\Delta_{DA}\right)\left(\mathcal D_1F\right)_{EBC} =&\frac12\left(\nabla_{A}\nabla_{D} +\Delta_{DA}\right)\left(\Delta_{BC}\cdot F_E -2\nabla_E\nabla_{(B}F_{C)}\right).
\end{aligned}\end{equation}}Noting  that in the  sum  of  (\ref{dFF})-(\ref{3}), terms involving  $F_A$ and $F_E$ are canceled each other, respectively, by the commutativity in (\ref{Delta}), we get
{\small\begin{equation*}\begin{aligned}
&4\left({\mathcal D}_2''\circ\mathcal D_1F\right)_{EDABC}\\=& \sum_{(D,B,C)} \left\{\bigg( \left(\nabla_E\nabla_D+\Delta_{DE}\right)\nabla_A\nabla_{(B} F_{C)} -\left(\nabla_A\nabla_D+\Delta_{DA}\right) \nabla_E\nabla_{(B} F_{C)}  \bigg)  +\frac12 \Delta_{BC}[\nabla_E,\nabla_A]F_D \right\}\\=& \sum_{(D,B,C)}\left\{ \bigg( \left(\nabla_E\nabla_D+\Delta_{DE}\right)\nabla_A\nabla_C -\left(\nabla_A\nabla_D+\Delta_{DA}\right) \nabla_E\nabla_C\bigg)F_B +\frac12\Delta_{DC}[\nabla_E,\nabla_A]F_B\right\} \\=& \sum_{(D,B,C)}\left\{\nabla_E\nabla_D\nabla_A\nabla_C -\nabla_E\Delta_{DA}\nabla_C-\nabla_A\nabla_D\nabla_E\nabla_C +\nabla_A\Delta_{DE}\nabla_C  +\frac12 \Delta_{DC}[\nabla_E,\nabla_A]\right\}F_B \\=&\sum_{(D,B,C)}\bigg(-\nabla_E\nabla_A\nabla_D\nabla_C +\nabla_A\nabla_E\nabla_D\nabla_C +\frac12\Delta_{DC}[\nabla_E,\nabla_A]\bigg) F_B\\=&\frac12\sum_{(D,B,C)}\bigg(-[\nabla_E,\nabla_A]\Delta_{DC} +\Delta_{DC}[\nabla_E,\nabla_A]\bigg) F_B=0,
\end{aligned}\end{equation*}}by using the first identity in  (\ref{symsym}) repeatedly.  Here in the third identity  we use anticommutator to move $\nabla_A,\nabla_E$ to the first two positions.  The theorem is proved.
\end{proof}

\section{The boundary complex}
\subsection{The boundary complex}
To prove Hartogs-Bochner extension for monogenic  functions in $k$ vector variables, it is necessary to find the first operator of the boundary complex of the Dirac complex (\ref{co3}). Let us recall the construction  of the boundary complex of a general differential complex on $\mathbb R^{kn} $ (see eg. \cite{Andreotti,AN,Naci,Nacinovich85}):
\begin{align*}
\Gamma\left(\mathbb R^{kn},E^{(0)}\right)\xrightarrow{A_0(x,\partial)} \Gamma\left(\mathbb R^{kn},E^{(1)}\right) \xrightarrow{A_1(x,\partial)}\Gamma\left(\mathbb R^{kn},E^{(2)}\right) \xrightarrow{A_2(x,\partial)}  \cdots.
\end{align*}
We say $u\in\Gamma_c\left(\overline \Omega,E^{(j)}\right)$   has \emph{zero Cauchy data} on the boundary $\partial \Omega$ for $A_j(x,\partial)$ if for any $\psi\in\Gamma_c\left(U,E^{(j+1)}\right),$ i.e. $\psi$ is   compactly supported in a given   open set $U,$ we have
\begin{align*}
\int_\Omega\left\langle A_j(x,\partial)u,\psi\right\rangle {\rm d}V=\int_\Omega\left\langle u,A_j^*(x,\partial)\psi\right\rangle {\rm d}V
\end{align*}
where $\langle\cdot,\cdot\rangle$ is the inner product of $E^{(j)},$ and $A_j^*(x,\partial)$ is the formal adjoint operator of $A_j(x,\partial).$
Given an open set $U.$ Set
\begin{align*}
\mathcal J_{A_j}(\partial \Omega,U):=\left\{u\in\Gamma\left(U,E^{(j)}\right);u\ {\rm has\ zero\ Cauchy\ date\ on}\ \partial \Omega\ {\rm for}\ A_j(x,\partial)\right\}.
\end{align*}
  Then $U\rightarrow \mathcal J_{A_j}(\partial \Omega,U)$ is a sheaf. We must have
\begin{align*}
A_j(x,\partial)\mathcal J_{A_j}(\partial \Omega,U)\subset\mathcal J_{A_{j+1}}(\partial \Omega,U).
\end{align*}
This is because
\begin{equation*}\begin{aligned}
\int_\Omega\left\langle A_{j+1}(x,\partial)A_{j}(x,\partial)u,\psi\right\rangle{\rm d}V=&0=\int_\Omega\left\langle u,A^*_{j}(x,\partial)A^*_{j+1}(x,\partial)\psi\right\rangle{\rm d}V\\=&\int_\Omega\left\langle A_{j}(x,\partial)u,A^*_{j+1}(x,\partial)\psi\right\rangle{\rm d}V,
\end{aligned}\end{equation*}
for any $u\in\mathcal J_{A_j}(\partial \Omega,U)$ and compactly supported $\psi\in\Gamma\left(U,E^{(j+1)}\right),$ by $A_{j+1}A_j=0,A^*_jA^*_{j+1}=0.$
Setting the quotient sheaf
\begin{align*}
Q^{(j)}(U)=\frac{\Gamma\left(U,E^{(j)}\right)}{\mathcal J_{A_j}\left(\partial \Omega,U\right)},
\end{align*}
we obtain a quotient complex of the form
\begin{align*}
Q^{(0)}(\partial \Omega)\xrightarrow{\widehat A_0} Q^{(1)}(\partial \Omega)\xrightarrow{\widehat A_1} Q^{(2)}(\partial \Omega)\xrightarrow{\widehat A_2}\dots
\end{align*}
where $\widehat A_j$ is induced by the differential operator $A_j(x, \partial),$ but is not necessarily a differential operator (cf. \cite{AN}).
\subsection{The formal adjoint operators}
Define the inner product on $\otimes^m\mathbb C^k\otimes \mathbb S^\pm$ as
\begin{equation*}
\langle f,G\rangle_{\otimes^m\mathbb C^k\otimes \mathbb S^\pm}:=\sum_{A_1,\dots,A_m=0}^{k-1} \left\langle f_{A_1\dots A_m}(x),G_{A_1\dots A_m}(x)\right \rangle_{\mathbb S^\pm},
\end{equation*} 
for  $f=f_{A_1\dots A_m}\omega^{A_1\dots A_m},G=G_{B_1\dots B_m}\omega^{B_1\dots B_m}\in  \otimes^m\mathbb C^k\otimes \mathbb S^\pm,$   where $\langle\cdot,\cdot\rangle_{\mathbb S^\pm}$ is the standard inner product on the spinor spaces $\mathbb S^\pm.$ This inner product induces naturally    inner products on subspaces $\mathcal V_j$ and $L^2\left(\mathbb R^{kn},\mathcal V_j\right).$ Let $\mathcal D_j^*$ be the formal adjoint operator of $\mathcal D_j,$  and   \begin{align*}
\hbox{d}V=\bigwedge_{j=0}^{k-1}(\hbox{d}x_{j1}\wedge\dots\wedge \hbox{d}x_{jn}),
\end{align*}
be the standard volume form on $\mathbb R^{kn}.$
\begin{lem}{\rm (1)}
For $G=\left(G_A\right)\in \Gamma_c \left(\mathbb R^{kn},\mathcal V_1\right),$ we have \begin{align}\label{D0star}\mathcal D_0^*G=\sum_{A=0}^{k-1}\nabla_AG_A.\end{align}  \\{\rm (2)} For  $h\in \Gamma_c \left(\mathbb R^{kn},\mathcal V_2\right),$ we have\begin{align}\label{D1star}\left(\mathcal D_1^*h\right)_C=\sum_{A,B=0}^{k-1} \left(\nabla_B\nabla_A h_{A(BC)}  -\frac12\Delta_{AB}h_{C A B}\right).\end{align}
\end{lem}
\begin{proof}
(1) For $f\in \Gamma_c \left(\mathbb R^{kn},\mathbb S^\pm\right),g\in \Gamma_c \left(\mathbb R^{kn},\mathbb S^\mp\right),$  we have
\begin{equation*}
\begin{aligned}
\int_{\mathbb R^{kn}}\left\langle\nabla_Af,g\right\rangle_{\mathbb S^\mp}\hbox{d}V =&\int_{\mathbb R^{kn}}\sum_i\left\langle\gamma_i\partial_{Ai}f, g\right\rangle_{\mathbb S^\mp}\hbox{d}V \\=&\int_{\mathbb R^{kn}}\sum_i\left\langle f, \gamma_i\partial_{Ai}g\right\rangle_{\mathbb S^\pm}\hbox{d}V =\int_{\mathbb R^{kn}}\left\langle f,{\nabla_Ag}\right\rangle_{\mathbb S^\pm}\hbox{d}V,
\end{aligned}\end{equation*}
for  $A=0,\dots,k-1,$ by integration by part and using $\gamma_j^*=-\gamma_j$ in Lemma \ref{gamma}. Thus the formal adjoint $\nabla_A^*$ of $\nabla_A$ satisfies
\begin{align}\label{nablastar}
\nabla_A^*=\nabla_A.
\end{align}Consequently, for $G=\left(G_A\right)\in \Gamma_c \left(\mathbb R^{kn},\mathcal V_1\right),$ we have \begin{equation*}
\begin{aligned}
\left(f,\mathcal D_0^*G\right)=\left(\mathcal D_0f,G\right)=\sum_{A=0}^{k-1}\int_{\mathbb R^{kn}}\left\langle\nabla_Af,{G_A}\right\rangle_{\mathbb S^-}\hbox{d}V=\int_{\mathbb R^{kn}}\left\langle f,\sum_{A=0}^{k-1}\nabla_AG_A \right\rangle_{\mathbb S^+}\hbox{d}V.
\end{aligned}\end{equation*}    (\ref{D0star}) follows. \\
(2)   For $F\in \Gamma_c \left(\mathbb R^{kn},\mathcal V_1\right),$   we have
{\small\begin{equation*}\begin{aligned}
\left(F,\mathcal D_1^* h\right)=&\left(\mathcal D_1F,h\right)=\int_{\mathbb R^{kn}}\left\langle \mathcal D_1F,h\right\rangle_{\mathcal V_2}\hbox{d}V\\=&\frac12\sum_{A,B,C=0}^{k-1 } \int_{\mathbb R^{kn}} \left\langle \nabla_{A}\nabla_{B} F_{C}+\nabla_{A}\nabla_{C} F_{B}-\Delta_{BC}F_{A},h_{ABC}\right\rangle_{ {\mathbb S^-}}\hbox{d}V
\\=&\frac12\sum_{A,B,C=0}^{k-1 }\int_{\mathbb R^{kn}} \bigg( \left\langle F_{C},\nabla_{B}\nabla_{A}h_{ABC}\right\rangle_{ {\mathbb S^-}}
+  \left\langle F_{B},\nabla_{C}\nabla_{A}h_{ABC}\right\rangle_{ {\mathbb S^-}}   - \left\langle F_{A},\Delta_{BC}h_{ABC}\right\rangle_{ {\mathbb S^-}} \bigg)\hbox{d}V
\\=&\sum_{C=0}^{k-1} \int_{\mathbb R^{kn}}\left\langle F_C,\sum_{A,B=0}^{k-1}\left(\nabla_B\nabla_A  h_{A(BC)}  -\frac12\Delta_{AB} h_{CAB}\right)\right\rangle_{\mathbb S^-}\hbox{d}V,
\end{aligned}\end{equation*}}by using (\ref{dF}), (\ref{nablastar}) and relabeling  indices. So we get (\ref{D1star}).
\end{proof}
\begin{cor}\label{p23}
For any $f\in \Gamma_c \left(\mathbb R^{kn},\mathcal V_0\right),$ we have
$
\mathcal D_0^*\mathcal D_0f=\Delta f,
$
where $\Delta=-\sum_{A=0}^{k-1}\sum_{j=1}^n\partial_{Aj}^2$ is the Laplacian operator on $\mathbb R^{kn}.$
\end{cor}
\begin{proof}
By (\ref{D0star}), we have
\begin{equation}\begin{aligned}\label{dsd}
\mathcal D_0^*\mathcal D_0f=\sum_{A=0}^{k-1}\nabla_A\nabla_Af=\sum_{A=0}^{k-1} \sum_{j,k=1}^{n} \gamma_j{\partial_{Aj}}\gamma_k{\partial_{Ak}}f=\Delta f.
\end{aligned}\end{equation}
The corollary  is proved.
\end{proof}
\begin{cor}\label{p34}
A  monogenic  function $f\in\mathcal O(\Omega)$ on a domain $\Omega\in\mathbb R^{kn}$ is real analytic.
\end{cor}
\begin{proof}
By (\ref{dsd}), we have $\mathcal D_0^*\mathcal D_0f=\Delta f=0,$
in the sense of distributions. Thus $f$ is harmonic and so it is  real analytic at each point of  $\Omega.$
\end{proof}
Without loss of generality, we can assume the defining function of the domain can be written locally as
 \begin{align}\label{defin}
\phi(\mathbf x)=x_{01}-\rho\left(x_{02},\dots x_{0n},\dots,x_{(k-1)n}\right),
 \end{align} with $\rho\left(0,\dots 0\right)=0$, for $\mathbf x\in \mathbb R^{kn}.$ Then
\begin{align*}
 Z_A=\nabla_A-\nabla_A\phi\left(\nabla_0\phi\right)^{-1}\nabla_0, \quad A=1,\cdots,k-1.
\end{align*}
are  tangential vector fieles, because $Z_A\phi=0,$ and  $$T=\left(\nabla_0\phi\right)^{-1} \nabla_0-\partial_{01},$$ is also tangential.
Then  \begin{equation}\begin{aligned}\label{nablaZ}
\nabla_0=&\nabla_0\phi \cdot T+\nabla_0\phi\cdot\partial_{01},\\ \nabla_A=& Z_A +\nabla_A\phi \cdot T+\nabla_A\phi\cdot\partial_{01}=\mathfrak Z_A+\nabla_A\phi\cdot\partial_{01},\quad  A=1,\dots,k-1,
\end{aligned}\end{equation}
where $$\mathfrak Z_A:=Z_A +\nabla_A\phi \cdot T$$ is the tangential part of $\nabla_A,$ i.e. it is a vector field tangential to hypersurfaces $\{\phi=t\}.$

We need the following coaera formula.
\begin{prop}{\rm(\cite[Theorem 1.2.4]{Mazja})}
For a measurable nonnegative function $u$ on an
open subset $\widetilde{\Omega}$ of $\mathbb R^{N}$ and $f\in C^1_c(\widetilde{\Omega}),$  we have
\begin{equation*}
\int_{\widetilde{\Omega}} u(x)|{\rm grad} f(x)| {\rm d}V(x)=\int_{0}^\infty{\rm d}t\int_{\widetilde{\Omega}\cap\{|f|=t\}}u(x) {\rm d}S(x),
\end{equation*}where ${\rm d} S$ is the $(N-1)$-dimension Hausdorff measure, which equals to the surface measure if the surface is smooth.
\end{prop}
Denote $\mathfrak Z_A^*$ be the formal adjoint of $\mathfrak Z_A$ on hypersurface $\{\phi={\rm constant}\}$ with respect to the measure $\frac{{\rm d}S}{|{\rm grad}\phi|},$ i.e.
\begin{equation*}\begin{aligned}
\int_{\{\phi=t\}}\mathfrak Z_Au\cdot v\frac{{\rm d}S}{|{\rm grad}\phi|}=\int_{\{\phi=t\}}u\cdot \mathfrak Z_A^*v\frac{{\rm d}S}{|{\rm grad}\phi|},
\end{aligned}\end{equation*}
for $u,v\in \Gamma_c(\mathbb R^{kn}).$\begin{prop} $\mathfrak Z_A^*=\mathfrak Z_A,$ i.e. the formal adjoint $\nabla_A^*$ of $\nabla_A$ can be written as
\begin{align}\label{nablaA}
\nabla_A^*=\mathfrak Z^*_A+\nabla_A\phi\cdot \partial_{01},\ A=0,\dots,k-1.
\end{align}
\end{prop}
\begin{proof}
 It follows from the coarea formula that  for $u,v\in \Gamma_c (\widetilde{\Omega} ),$
  \begin{equation}\begin{aligned}\label{Zstar}
\int_{\widetilde{\Omega}}  \mathfrak Z_Au\cdot v {\rm d}V=&\int_0^\infty{\rm d}t\int_{\widetilde{\Omega}\cap\{\phi=t\}}\mathfrak Z_Au\cdot v\frac{{\rm d}S}{|{\rm grad}\phi|}\\=&\int_{0}^\infty{\rm d}t\int_{\widetilde{\Omega}\cap\{\phi=t\}}u\cdot \mathfrak Z_A^*v\frac{{\rm d}S}{|{\rm grad}\phi|}=\int_{\widetilde{\Omega}}   u\cdot \mathfrak Z_A^*v {\rm d}V.
\end{aligned}\end{equation}
Here we can modify $\phi$ outside of support of $u$ and $v$ so that it is compactly supported.
Then for $f\in \Gamma_c\left(\mathbb R^{kn},  \mathbb S^+\right)$ and $h=\left(h_A\right)\in \Gamma_c\left(\mathbb R^{kn},\mathcal V_1\right),$ since $\mathfrak Z_A$ is a scalar operator and  \begin{align}\label{nablaphi}\left(\nabla_A\phi\right)^* =-\nabla_A\phi\end{align} by $\gamma_j^*=-\gamma_j$, we have
\begin{equation*}
\begin{aligned}
\int_{\mathbb R^{kn}}\left\langle\nabla_Af,{h_A}\right\rangle_{\mathbb S^-}\hbox{d}V =&\int_{\mathbb R^{kn}}\left\langle\mathfrak Z_Af,{h_A}\right\rangle_{\mathbb S^-}\hbox{d}V +\int_{\mathbb R^{kn}}\left\langle\mathbf \nabla_A\phi\cdot\partial_{01}f,{h_A}\right\rangle_{\mathbb S^-}\hbox{d}V\\=&\int_{\mathbb R^{kn}}\left\langle  f,\mathfrak Z_A^*{h_A}\right\rangle_{\mathbb S^+}\hbox{d}V +\int_{\mathbb R^{kn}}\left\langle f,\partial_{01}\left(\nabla_A\phi{h_A}\right) \right\rangle_{\mathbb S^-}\hbox{d}V,
\end{aligned}\end{equation*}
by (\ref{nablaZ}) and  (\ref{Zstar}), for  $A=0,\dots,k-1.$ Since we have known $\nabla^*_A=\nabla_A$, and \begin{align}\label{partial01}\partial_{01}\left(\nabla_A\phi h_A\right)=\nabla_A\phi\partial_{01} h_A,\end{align} by $\partial_{01}\nabla_A\phi=\nabla_A\partial_{01}\phi=\nabla_A1=0,$  we get $\mathfrak Z_A^*=\mathfrak Z_A.$  The proposition is proved.
\end{proof}
 We need the following formula of integration by part for $\nabla_A.$
\begin{prop}Suppose that $\Omega$ is a bounded domain in $\mathbb R^{kn}.$ For any $f\in C^1\left(\overline \Omega,\mathbb S^\pm\right)$ and $h\in C^1\left(\overline \Omega,\mathbb S^\mp\right),$ we have
\begin{equation}\begin{aligned}\label{bp}
\int_{\Omega}\left\langle\nabla_Af,h\right\rangle_{\mathbb S^\mp}{\rm d}V=&\int_\Omega\left\langle f, \nabla_A h\right\rangle_{\mathbb S^\pm}{\rm d}V -\int_{\partial \Omega}\left\langle \nabla_A\phi\cdot f, h \right\rangle_{\mathbb S^\mp}\frac{{\rm d}S}{|{\rm grad}\phi|}.
\end{aligned}\end{equation}
\end{prop}
\begin{proof}Recall that  $\mathbf n=-\frac{{\rm grad}  \phi}{|{\rm grad} \phi|}$ is the unit vector outer normal to $\partial \Omega$ with ${\rm grad}\ \phi(\mathbf x)=\left(1,-\rho_{x_{02}},-\rho_{x_{03}}, \dots\right).$ It is easy to see that
\begin{equation*}\begin{aligned}
\int_\Omega\left\langle\nabla_Af,h\right\rangle_{\mathbb S^\mp}{\rm d}V=&\int_\Omega\left\langle \mathfrak Z_A f, h\right\rangle_{\mathbb S^\mp}{\rm d}V+\int_\Omega\left\langle\nabla_A\phi\cdot\partial_{01}f, h\right\rangle_{\mathbb S^\mp}{\rm d}V\\=&\int_\Omega\left\langle f, \mathfrak Z_A^* h\right\rangle_{\mathbb S^\pm}{\rm d}V+\int_\Omega\left\langle f, \partial_{01}\left(\nabla_A\phi\cdot h\right)\right\rangle_{\mathbb S^\pm}{\rm d}V +\int_{\partial \Omega}\left\langle f, \nabla_A\phi\cdot h \right\rangle_{\mathbb S^\pm}\frac{{\rm d}S}{|{\rm grad}\phi|}\\=&\int_\Omega\left\langle f, \nabla_A^* h\right\rangle_{\mathbb S^\pm}{\rm d}V-\int_{\partial \Omega}\left\langle\nabla_A\phi\cdot f, h \right\rangle_{\mathbb S^\pm}\frac{{\rm d}S}{|{\rm grad}\phi|},
\end{aligned}\end{equation*}by (\ref{Zstar})-(\ref{partial01}).
The proposition is proved.
\end{proof}

\subsection{The characterization of the sheafs $\mathcal J_j(U)$'s }
\begin{lem}\label{exp}
 $H \in \Gamma(U,\mathcal V_1)$ satisfies \begin{align}\label{JJJ1} \sum_{(B,C)}\nabla_{[A}\phi\nabla_{\underline{B}}\phi H_{C]}=0,\quad {\rm on}\ \partial\Omega\cap U,\end{align} for any $A,B,C\in\{0,\dots,k-1\},$ if and only if there exists some  $H^{(1)}\in \Gamma(U,\mathbb S^+) $ such that
\begin{align}\label{Hc}
H_C=\nabla_C\phi \cdot H^{(1)}+ 
O\left(\phi\right).
\end{align} \end{lem}
\begin{proof}
Letting $B=C=0$ in   (\ref{JJJ1}),   we get\begin{align*}
 \nabla_A\phi\nabla_0\phi\cdot H_0 = \left(\nabla_0\phi\right)^2\cdot H_A,\quad {\rm on}\ \partial\Omega.
\end{align*}
Note that  \begin{align}\label{id}\left(\nabla_0\phi\right)^2 =\left|\nabla_0\phi\right|^2{\rm Id}_{\mathbb S^\pm}\end{align} commutes $\nabla_A\phi$ and $\left|\nabla_0\phi\right|$ obviously does not vanish by definition (\ref{defin}) of $\phi$. If denote $H^{(1)}=\left.\nabla_0\phi \left(\nabla_0 \phi\right)^{-2}H_0\right|_{\partial \Omega},$   we have
\begin{align*}
 H_A  = \nabla_A\phi\cdot\nabla_0\phi \left(\nabla_0 \phi\right)^{-2}H_0 =\nabla_A\phi\cdot H^{(1)}, \quad {\rm on}\ \partial\Omega,
\end{align*}
for   $A=0,\dots,k-1.$ Thus (\ref{Hc}) follows.

Conversely, $H_C$ given by (\ref{Hc}) must  satisfy (\ref{JJJ1}) since
\begin{equation}\begin{aligned}\label{comm} 2\sum_{(B,C)}\nabla_{[A}\phi \nabla_{\underline{B}}\phi\nabla_{C]}\phi =&\nabla_A\phi\sum_{(B,C)}\nabla_B\phi\nabla_C\phi- \sum_{(B,C)}\nabla_C\phi\nabla_B\phi \nabla_A\phi\\=& \left[\nabla_A\phi,\left\{\nabla_B\phi, \nabla_C\phi\right\}\right]=0. \end{aligned}\end{equation}
This is because
$$\left\{\nabla_B\phi,\nabla_C\phi\right\} =\sum_{i,j=1}^n \left(\gamma_i\gamma_j\partial_{Bi}\phi\partial_{Cj}\phi +\gamma_j\gamma_i\partial_{Cj}\phi\partial_{Bi}\phi\right) =-2\sum_{i=1}^n\partial_{Bi}\phi\partial_{Ci}\phi\cdot{\rm Id}_{\mathbb S^\pm}.$$   The lemma is proved.
\end{proof}

\begin{lem}\label{nana}
For $F\in\Gamma\left(U,\mathbb S^\pm\right),$ we have
{ \begin{equation}\begin{aligned}\label{rhog}
0=&\sum_{(B,C)} \bigg(\nabla_{[A}\nabla_{\underline{B}}\left(\phi\nabla_{C]}\phi \cdot F\right)+\nabla_{[A}\phi\nabla_{\underline{B}}\phi\nabla_{C]} F \bigg) +O(\phi)\\=&\sum_{(B,C)} \bigg(
\nabla_{[A}\left(\nabla_{\underline{B}}\phi\nabla_{C]}\phi \cdot F\right)+\nabla_{[A}\phi\nabla_{\underline{B}}\left(\nabla_{C]}\phi \cdot F\right)+\nabla_{[A}\phi\nabla_{\underline{B}}\phi \nabla_{C]} F\bigg) +O(\phi).
\end{aligned}\end{equation}}
\end{lem}
\begin{proof}Since $ 2\sum_{(B,C)}\nabla_{[A} \nabla_{\underline{B}} \nabla_{C]}= \left[\nabla_A,\Delta_{BC}\right]=0,$ we have
\begin{equation}\begin{aligned}\label{phi2}
\sum_{(B,C)}\nabla_{[A}\nabla_{\underline{B}}\nabla_{C]} \left(\phi^2F\right)=0.
\end{aligned}\end{equation}
On the other hand, for $F\in\Gamma\left(U,\mathbb S^\pm\right),$
{\small\begin{equation}\begin{aligned}\label{351}
\nabla_A\nabla_B\nabla_C\left(\phi^2F\right) =&\nabla_A\nabla_B\left(2\phi\nabla_C\phi \cdot F\right)+\nabla_A\nabla_B\left(\phi^2\nabla_CF\right)
\\=&2\nabla_A\bigg(\nabla_B \phi\nabla_C\phi\cdot F+\phi\nabla_B\left(\nabla_C\phi \cdot F\right)\bigg)+\nabla_A \bigg(2\phi\nabla_B\phi\nabla_CF+\phi^2\nabla_B\nabla_C F\bigg)
\\=&2\nabla_A\left(\nabla_B\phi\nabla_C\phi \cdot F\right)+2\nabla_A\phi\nabla_B\left(\nabla_C\phi \cdot F\right)+2\nabla_A\phi\nabla_B\phi \nabla_C F +O(\phi).
\end{aligned}\end{equation}}Then   (\ref{rhog}) follows from (\ref{phi2}) and skewsymmetrization of  $A$ and $C$ in (\ref{351}).
\end{proof}

\begin{prop}\label{J13}{\rm (1)}
$\mathcal J_0(U)$ consists of $\mathbb S^+$-valued function $f$ such that $f=O(\phi).$\\
{\rm (2)} $\mathcal J_1(U)$ consists of $f= f_C\omega^C$ with \begin{equation}\begin{aligned}\label{J3'}
f_C=\nabla_C\phi\cdot F+\phi \left(\nabla_C F+\nabla_C\phi \cdot F'\right)+O\left(\phi^2\right)=\nabla_C\left(\phi F+\phi^2\cdot\frac{F'}{2}\right)+O\left(\phi^2\right),
\end{aligned}\end{equation}
 for some  $F,F'\in\Gamma\left(U,\mathbb S^+\right).$
\end{prop}
\begin{proof} (1) By definition,  $\mathcal J_0(U)$ is the set of  all $f\in\Gamma\left(U,\mathcal V_0\right)$ satisfying
\begin{align*}
 \int_\Omega\left\langle \mathcal D_0f ,G\right\rangle_{\mathbb S^-}{\rm d}V=\int_\Omega\left\langle f,\mathcal D_0^*G \right\rangle_{\mathbb S^+}{\rm d}V,
\end{align*}for any $G\in\Gamma_c\left(U,\mathcal V_1\right)$. We have
\begin{equation*}\begin{aligned}
\left(\mathcal D_0f,G\right)_{\mathcal V_1}=&\sum_{A=0}^{k-1}\int_\Omega\left\langle \nabla_A f,G_A\right\rangle_{\mathbb S^-} \hbox{d}V
\\=&\int_\Omega\left\langle f,\mathcal D_0^*G\right\rangle_{\mathbb S^+} {\rm d}V-\sum_{A=0}^{k-1}\int_{\partial \Omega}\left\langle\nabla_A\phi\cdot  f, G_A\right\rangle_{\mathbb S^-} \frac{{\rm d}S}{|{\rm grad}\phi|},
\end{aligned}\end{equation*}by  using (\ref{D0star}) and integration by part (\ref{bp}). Since $G_A$  are arbitrarily  chosen, $f\in\mathcal J_0(U)$  if and only if
$
\left.\nabla_A\phi\cdot f\right|_{\partial \Omega}=0.
$
Taking $A$ such that $\nabla_A\phi\neq0$ in a neighbourhood, we have  $\left.\left(\nabla_A\phi \right)^2 \cdot f\right|_{\partial \Omega}=0.$  But by (\ref{id}), we see  $\left.f\right|_{\partial \Omega\cap U}=0.$ So $f=O(\phi).$\\
(2) By definition, $\mathcal J_1(U)$ is the set of all  $f\in\Gamma\left(U,\mathcal V_1\right)$ satisfying
\begin{align*}
 \int_\Omega\left\langle \mathcal D_1f,h \right\rangle_{\mathbb S^-}{\rm d}V= \int_\Omega\left\langle f,\mathcal D_1^*h \right\rangle_{\mathbb S^-}{\rm d}V,
\end{align*}for any $h\in\Gamma_c\left(U,\mathcal V_2\right)$. For  $g \in \Gamma_c\left(U,\mathbb S^-\right)$ with compact support, we have
 \begin{equation*}\begin{aligned}
\int_\Omega\left\langle\nabla_A\nabla_Bf_C,g \right\rangle_{\mathbb S^-}{\rm d}V=&\int_\Omega\left\langle \nabla_Bf_C,\nabla^*_Ag\right\rangle_{\mathbb S^+}{\rm d}V-\int_{\partial \Omega}\left\langle\nabla_A\phi\nabla_Bf_C, g\right\rangle_{\mathbb S^+}\frac{{\rm d}S}{|{\rm grad}\phi|}
\\=&\int_\Omega\left\langle f_C,\nabla^*_B\nabla_A^*g\right\rangle_{\mathbb S^-}{\rm d}V-\int_{\partial \Omega}\left\langle\nabla_B\phi f_C, \left(\mathfrak Z_A^*+\nabla_A\phi\cdot \partial_{01}\right) g\right\rangle_{\mathbb S^-}\frac{{\rm d}S}{|{\rm grad}\phi|}\\&\qquad\qquad\qquad\qquad\qquad-\int_{\partial \Omega}\left\langle \nabla_A\phi\nabla_B f_C,g\right\rangle_{\mathbb S^-}\frac{{\rm d}S}{|{\rm grad}\phi|}
\\=&\int_\Omega\left\langle f_C,\nabla_B \nabla_A  g\right\rangle_{\mathbb S^-}{\rm d}V+\int_{\partial \Omega}\left\langle \nabla_A\phi \nabla_B\phi f_C,\partial_{01}g\right\rangle_{\mathbb S^-}\frac{{\rm d}S}{|{\rm grad}\phi|}\\&\qquad\qquad\qquad\qquad\quad-\int_{\partial \Omega}\left\langle \mathfrak Z_A\left(\nabla_B\phi f_C\right)+\nabla_A\phi \nabla_Bf_C,g\right\rangle_{\mathbb S^-}\frac{{\rm d}S}{|{\rm grad}\phi|},
\end{aligned}\end{equation*} by  using (\ref{nablaA}), integration by part  (\ref{bp}) twice and (\ref{nablaphi}). After skewsymmetrization of $A$ and $C$ and symmetrization $B$ and $C$ above, then apply it to $g=h_{ABC}$ such that $\left(h_{ABC}\right)\in\Gamma\left(U,\mathcal V_2\right)$ and take summation over $A,B,C$ to get
{\small\begin{equation*}\begin{aligned}
\left(\mathcal D_1f,h\right)_{\mathcal V_2}=&\int_\Omega\left\langle \nabla_{[A}\nabla_{\underline{B}}f_{C]} +\nabla_{[A}\nabla_{\underline{C}}f_{B]},h_{ABC} \right\rangle_{\mathbb S^-} \hbox{d}V\\=&\int_\Omega\left\langle f,\mathcal D_1^*h\right\rangle_{\mathbb S^-} {\rm d}V+\sum_{A,B,C}\int_{\partial \Omega}\bigg(\left\langle \mathcal S_{ABC},\partial_{01}h_{ABC}\right\rangle_{\mathbb S^-}  - \left\langle  \mathcal S'_{ABC},h_{ABC}\right\rangle_{\mathbb S^-}\bigg) \frac{{\rm d}S}{|{\rm grad}\phi|}\\=&\left(f,\mathcal D_1^*h\right)_{\mathcal V_1}+\int_{\partial \Omega}\bigg(\left\langle \mathcal S,\partial_{01}h\right\rangle_{\mathcal V_2}   - \left\langle {\mathcal S'},h\right\rangle_{\mathcal V_2} \bigg) \frac{{\rm d}S}{|{\rm grad}\phi|},
\end{aligned}\end{equation*}}where $ \mathcal D_1^*h $ is given by (\ref{D1star}), $\mathcal S =(\mathcal S_{ABC})$ and $\widetilde{\mathcal S }=(\widetilde{\mathcal S}_{ABC})$ with
\begin{equation*}\begin{aligned}
\mathcal S_{ABC}:=&\sum_{(B,C)} \left.\bigg(\nabla_{[A}\phi\nabla_{\underline{B}} \phi f_{C]}  \bigg)\right|_{\partial \Omega},\\  \mathcal S'_{ABC}:=&\sum_{(B,C)}\left.\bigg(\mathfrak Z_{[A}\left(\nabla_{\underline{B}}\phi f_{C]}\right)  +\nabla_{[A}\phi \nabla_{\underline{B}}f_{C]}\bigg) \right|_{\partial \Omega},
\end{aligned}\end{equation*} $A,B,C=0,\dots,k-1.$ Here $\mathcal S \in\Gamma\left(\partial\Omega,\mathcal V_2\right)$ since $\mathcal S=\frac32C_{21} (\widetilde S )$ with $\widetilde S_{ABC}=\nabla_A\phi\nabla_B\phi f_C.$ Similarly, $\mathcal S'\in\Gamma\left(\partial\Omega,\mathcal V_2\right).$ Since $h$ and $\partial_{{01}}h$ are arbitrarily chosen $\mathcal V_2$-valued functions, $f\in\mathcal J_1(U)$  if and only if $\mathcal S=0,\mathcal S'=0.$  By Lemma \ref{exp}, $\mathcal S=0$ holds if and only if
\begin{align}\label{f}
f_C=\nabla_C\phi \cdot F+\phi G_C+O\left(\phi^2\right),
\end{align}
for some  $F\in \Gamma(U,\mathbb S^+) $ and $G_C\in \Gamma(U,\mathbb S^-),$ where   $C=0,\dots,k-1.$
 Then substitute  (\ref{f}) into $\mathcal S'=0$ to get
\begin{equation}\begin{aligned}\label{JJJ2'}
0=&\sum_{(B,C)}\bigg\{\mathfrak Z_{[A}\left(\nabla_{\underline{B}}\phi \left(\nabla_{C]}\phi\cdot F\right)\right)      +\nabla_{[A}\phi \nabla_{\underline{B}}\left(\nabla_{C]}\phi\cdot F\right) +\nabla_{[A}\phi \nabla_{\underline{B}}\left(\phi G_{C]} \right)\bigg\}+O(\phi) \\=&:\Sigma_1+\Sigma_2+\Sigma_3+O(\phi),
\end{aligned}\end{equation}by $\mathfrak Z_{A}\phi=0$ for any $A,$
where \begin{equation*}\begin{aligned}
\Sigma_3=&\sum_{(B,C)} \nabla_{[A}\phi \nabla_{\underline{B}}\left(\phi G_{C]} \right)   = \sum_{(B,C)} \nabla_{[A}\phi \nabla_{\underline{B}}\phi\cdot G_{C]}     +O(\phi),\\ \Sigma_1 +\Sigma_2 =&\sum_{(B,C)}\bigg\{\mathfrak Z_{[A}\left(\nabla_{\underline{B}}\phi\cdot \nabla_{C]}\phi\cdot F\right)     +\nabla_{[A}\phi \nabla_{\underline{B}}\left(\nabla_{C]}\phi\cdot F\right)\bigg\} \\=&\sum_{(B,C)}\bigg\{\nabla_{[A}\left(\nabla_{\underline{B}}\phi \cdot \nabla_{C]}\phi\cdot F\right)  -\nabla_{[A}\phi \nabla_{\underline{B}}\phi \nabla_{C]}\phi\cdot\partial_{01} F   +\nabla_{[A}\phi \nabla_{\underline{B}}\left(\nabla_{C]}\phi\cdot F\right)\bigg\}\\=&\sum_{(B,C)}\bigg( \nabla_{[A}\left(\nabla_{\underline{B}}\phi \cdot \nabla_{C]}\phi\cdot F\right)     +\nabla_{[A}\phi \nabla_{\underline{B}}\left(\nabla_{C]}\phi\cdot F \right)\bigg) +O(\phi).
\end{aligned}\end{equation*}Here the second identity holds by (\ref{partial01})  and  the last identity holds by (\ref{comm}).
By the second identity of (\ref{rhog}) in Lemma \ref{nana}, we have
\begin{equation*}\begin{aligned}
0=
\Sigma_1+\Sigma_2+ \sum_{(B,C)}\nabla_{[A}\phi\nabla_{\underline{B}}\phi \nabla_{C]}F+O(\phi).
\end{aligned}\end{equation*}
Then   (\ref{JJJ2'}) is equivalent to
\begin{align}\label{325}
\sum_{(B,C)} \nabla_{[A}\phi\nabla_{\underline{B}}\phi \left(G_{C]}-\nabla_{C]}F\right) +O(\phi)=0.
\end{align}
Then apply Lemma \ref{exp} again to $H_C=G_C-\nabla_CF$ to get
$
G_C-\nabla_CF=\nabla_C \phi \cdot F' +O\left(\phi\right),
$
for some $F'\in\Gamma\left(U,\mathbb S^-\right)$. The proposition is proved.
\end{proof}
\subsection{The tangential several Dirac operators}By the characterization of $\mathcal J_0(U)$ and $\mathcal J_1(U)$ in Proposition \ref{J13}, we have isomorphism  \begin{equation*}\begin{aligned}
\pi_0:\Gamma\left(U,\mathcal V_0\right)/\mathcal J_0(U)&\xrightarrow{\cong} \Gamma\left(U\cap\partial\Omega,\mathscr V_0\right),\\f=\hat f+\phi\hat{f'}+O\left(\phi^2\right)&\mapsto\hat f,
\end{aligned}\end{equation*}for $\hat f,\hat{f'}$ independent of $x_{01},$ where $\mathscr V_0=\mathcal V_0,$  and \begin{equation*}\begin{aligned}
\pi_1:\Gamma\left(U,\mathcal V_1\right)/\mathcal J_1(U)&\xrightarrow{\cong} \Gamma\left(U\cap\partial\Omega,\mathscr V_1\right),\end{aligned}\end{equation*} with $\mathscr V_1\cong\left(\mathbb S^-\otimes\mathbb C^{k-1}\right)\oplus\left(\mathbb S^-\otimes\mathbb C^{k-1}\right),$ given by $$F=\widehat F+\phi\widehat{F}'+O\left(\phi^2\right)\mapsto\left(F_1,F_2\right),$$
for $\widehat F,\widehat F'$ independent of $x_{01,}$ where \begin{equation}\label{eq:F12}\begin{aligned}
\left(F_1\right)_\mu=&\widehat F_\mu-\nabla_\mu\phi\left(\nabla_0\phi\right)^{-1}\widehat F_0,\\\left(F_2\right)_\mu=&\widehat{F}'_\mu-\nabla_\mu\left( \left(\nabla_0\phi\right)^{-1}\widehat F_0\right)-\nabla_\mu\phi\left(\nabla_0\phi\right)^{-1} \left(\widehat{F}'_0-\nabla_0\left(\left(\nabla_0\phi\right)^{-1} \widehat F_0\right)\right),
\end{aligned}\end{equation}
for $\mu=1,2,\dots,k-1.$ Note that for $\mu=0,$ $\left(F_1\right)_0=\left(F_2\right)_0=0$ by definition.  $\pi_1$ is well-defined  because for $F_A=\nabla_A\left(\phi F+\phi^2\frac{F'}{2}\right)+O\left(\phi^2\right)\in \mathcal J_1(U),$   if we take  $\widehat F_A=\nabla_A\phi\cdot F,$ $\widehat{F}'_A=\nabla_A  F+\nabla_A\phi\cdot F',$ it is easy to check that $\pi_1(F)=0 $ by \eqref{eq:F12}. Then we call the operator  induced from several Dirac operators the \emph{tangential several  Dirac operators}:
\begin{equation*}\begin{aligned} {\mathscr D}_0:\Gamma\left(U\cap\partial \Omega,\mathscr V_0\right) &\rightarrow \Gamma\left(U\cap\partial\Omega,\mathscr V_1\right),\\\hat f&\mapsto\pi_1\left(\mathcal D_0\hat f\right)=\left(F_1,F_2\right). \end{aligned}\end{equation*}
It is given by \begin{equation}\begin{aligned}\label{f12} \left(F_1\right)_\mu=&\nabla_\mu\hat f-\nabla_\mu\phi\left(\nabla_0\phi\right)^{-1}\nabla_0\hat f=Z_\mu\hat f,\\\left(F_2\right)_\mu=&-\nabla_\mu\left( \left(\nabla_0\phi\right)^{-1} \nabla_0\hat f\right)+\nabla_\mu\phi\left(\nabla_0\phi\right)^{-1} \nabla_0\left(\left(\nabla_0\phi\right)^{-1}\nabla_0\hat f\right)=-Z_\mu T\hat f,\end{aligned}\end{equation}
by $T\hat f=\left(\nabla_0\phi\right)^{-1} \nabla_0\hat f,$ since $\hat f$ is independent of $x_{01,}$ for $\mu=1,\dots,k-1.$ On a domain $D$
in $\partial \Omega,$ a function $\hat f\in\Gamma\left(D,\mathcal V_0\right)$ is called  \emph{tangentially  monogenic}  if  $ {\mathscr D}_0\hat f=0,$  i.e. $  Z_\mu \hat f=0$ and $ Z_\mu T \hat f=0.$ Equivalently, $Z_\mu \hat f=0$ and $\left[Z_\mu,T\right]\hat f=0$   as a system of first order differential equations.

\begin{prop}\label{phat}
If $\hat f\in C^\infty\left(\partial \Omega,\mathcal V_0\right)$ is tangentially monogenic,  there exists a representative $\tilde  f\in\Gamma\left(\overline \Omega,\mathcal V_0\right)$ such that $\left.\tilde f\right|_{\partial \Omega}=\hat f$ and $\mathcal D_0\tilde f$ is flat on $\partial \Omega.$
\end{prop}
\begin{proof}
Since $\hat f$ is tangentially monogenic, $ {\mathscr D}_0\hat f=0,$ i.e. $Z_A\hat f=0, Z_AT\hat f=0.$ Let $\tilde f$ be an smooth extension of  $\hat f$ to a neighborhood of $\overline \Omega,$ we see that  $\pi_1 \mathcal D_0\tilde f )=0$ by (\ref{f12}), i.e. $\mathcal D_0\tilde f\in\mathcal J_1(U)$ for any open set $U.$ Thus
\begin{align*}
\nabla_C\tilde f =\nabla_C\left(\phi  \tilde F+\phi^2\cdot\frac{\tilde F'}{2}\right)  +O\left(\phi^2\right),
\end{align*}
for some $\mathbb S^-$-valued  function $\tilde F,\tilde F',$   by  Proposition \ref{J13}. Then we have \begin{align*}
\nabla_C(\tilde f-\phi \tilde F)=\phi\nabla_C \phi\cdot\tilde F'+\phi^2\tilde G_C+O\left(\phi^3\right),
\end{align*}for some $\mathbb S^-$-valued  function  $\tilde G_C.$
  As $\mathcal D_1\circ\mathcal D_0(\tilde f-\phi \tilde F)=0,$   we get
\begin{equation*}\begin{aligned}
0=& \sum_{(B,C)}\nabla_{[A} \nabla_{\underline{B}} \left(\phi\nabla_{C]}\phi\cdot\tilde F'\right)+\sum_{(B,C)}\nabla_{[A} \nabla_{\underline{B}}\left(\phi^2\tilde G_{C]}\right)  +O(\phi) \\=& \sum_{(B,C)}\nabla_{[A} \phi\nabla_{\underline{B}}\phi\left(2\tilde G_{C]}-\nabla_{C]}\tilde F'\right) +O(\phi),
\end{aligned}\end{equation*}
by the first identity of (\ref{rhog}) in Lemma \ref{nana} and
\begin{align*}
\nabla_A\nabla_B\left(\phi^2 \tilde G_C\right)=2\nabla_A\phi\nabla_B\phi\tilde G_C+O(\phi).
\end{align*}
 Then by using Lemma \ref{exp},  we have
\begin{align*}
2\tilde G_C=\nabla_C\tilde F'+\nabla_C\phi \cdot\tilde h'+O\left(\phi\right),
\end{align*}
for some $\mathbb S^-$-valued function  $\tilde h'$.  Therefore \begin{align*}
\nabla_C(\tilde f-\phi \tilde F)=\phi\left[\nabla_C\phi\cdot \tilde F'+\frac\phi2\left(\nabla_C\tilde F'+\nabla_C\phi\cdot \tilde h' \right)+O\left(\phi^2\right)\right],
\end{align*}
and so
\begin{align*}
\nabla_C\left(\tilde f-\phi \tilde F-\frac{\phi^2}{2} \tilde F'\right)=\frac{\phi^2}{2} \nabla_C\phi \cdot\tilde h'+O\left(\phi^3\right).
\end{align*}
  Repeating this procedure, we get
\begin{align*}
\nabla_C\left(\tilde f+\phi \tilde f^{(1)}+\phi^2\tilde f^{(2)}+\dots\right)=O_{\partial \Omega}^\infty
\end{align*}
with a formal power series in $\phi$ with coefficients $C^\infty$ on $\partial \Omega,$ where $O_{\partial \Omega}^\infty$ denotes functions vanishing of infinite order on $\partial \Omega.$ By using Whitney extension theorem as in  \cite{AH1,AH2} we get the conclusion.
The proposition is proved.
\end{proof}
\begin{rem}
{\rm (1)} Such extension for CR functions was constructed by Andreotti-Hill \cite{AH1,AH2}, while extension for pluriharmonic functions satisfying $\partial\bar\partial$-equation was constructed by Andreotti-Nacinovich in \cite{AN}. Recently, such extension for $k$-CF functions was constructed by the second author in \cite{wang29}.\\
{\rm(2)}   It is sufficient to prove Theorem \ref{HBE} that there exists an extension   $\tilde f\in C^3\left(\overline\Omega\right)$ such that $\mathcal D_0 \tilde f$ vanishes on $\partial \Omega$ up to derivatives of second order.
\end{rem}

 \section{Elliptic  complex and solutions of non-homogeneous several Dirac equations}
 \subsection{The ellipticity of the  short Dirac complex}
A differential complex
\begin{equation}\label{seq}
0\rightarrow \Gamma\left(\Omega,V_0\right)\xrightarrow{\mathcal{D}_{0}}\cdots \xrightarrow{\mathcal{D}_{j-1}} \Gamma\left(\Omega,V_j\right)\rightarrow \cdots,
\end{equation} is called \emph{elliptic} if its symbol sequence
\begin{equation*}
0\rightarrow V_0 \xrightarrow{\sigma\left(\mathcal{D}_{0}\right){(\mathbf x;  \xi)}}\cdots \xrightarrow{\sigma\left(\mathcal{D}_{j-1}\right){(\mathbf x;  \xi)}} V_j\rightarrow \cdots,
\end{equation*} is exact for any $\mathbf x\in \Omega,  \xi\in\mathbb R^{kn}\setminus\{\mathbf 0\},$ i.e. $$\ker\sigma\left(\mathcal{D}_{l}\right){(\mathbf x;  \xi)}={\rm Im}\ \sigma\left(\mathcal{D}_{l-1}\right){(\mathbf x;  \xi)}.$$
The  symbol of differential operator $\nabla_A$ is
 \begin{align}\label{simp}
\sigma\left(\nabla_A\right)(\mathbf x;  \xi)=-{\mathbf i}\sum_{j=1}^n {\gamma_j\xi_{A j}}:=\boldsymbol \xi_A,
\end{align}
which is a linear transformation   $\mathbb S^\pm\rightarrow\mathbb S^\mp.$

  Denote  $\left|\boldsymbol \xi_A\right|^2:=\xi_{A1}^2+\dots+\xi_{An}^2$ and $ \boldsymbol \xi_{BC}=:\boldsymbol \xi_B\boldsymbol \xi_C+\boldsymbol \xi_C\boldsymbol \xi_B.$ Then
\begin{align}\label{nunu}
\boldsymbol \xi_A\boldsymbol \xi_A=-{\mathbf i}\sum_k {\gamma_k\xi_{A k}}\left(-{\mathbf i}\sum_j{\gamma_j\xi_{A j}}\right)=\left|\boldsymbol \xi_A\right|^2{\rm id}_{\mathbb S^\pm}.
\end{align}

\begin{lem}\label{l51} For $\Theta\in \ker \sigma'_2(  \xi)\cap\ker \sigma''_2(  \xi)$, we have  \begin{equation}\begin{aligned}\label{tabc}
 \left|\boldsymbol \xi_0\right|^2\Theta_{ABC}= \boldsymbol \xi_A\boldsymbol \xi_B\Theta_{00C} +\boldsymbol \xi_A\boldsymbol \xi_C\Theta_{00B}-  \xi_{BC}\Theta_{00A}.
\end{aligned}\end{equation}
\begin{proof}
Firstly, note that since $\Theta\in\mathcal V_2,$ we have  \begin{equation}\begin{aligned}\label{TABC}
\Theta_{[A\underline{B}C]}+\Theta_{[A\underline{C}B]} =\frac32\Theta_{ABC},\quad\Theta_{ABC}=\Theta_{ACB},
\end{aligned}\end{equation}by  the characterization of $\mathcal V_2$ in (\ref{HABC})-(\ref{Hcomm}). It follows from
  $\Theta\in \ker \sigma'_2(  \xi)$ that  \begin{equation}\begin{aligned}\label{sig2} \left(\sigma'_2(  \xi)\Theta\right)_{DABC}
=\frac32\left(\boldsymbol \xi_D\Theta_{ABC}+\boldsymbol \xi_A\Theta_{DBC} + \boldsymbol \xi_B\Theta_{CDA} +\boldsymbol \xi_C\Theta_{BDA}\right) =0,\end{aligned}\end{equation}  as in  (\ref{d2d1}). If letting $B=C=0$ in the first identity in  (\ref{TABC}), we get
\begin{align}\label{t00}
\Theta_{A00}=-2\Theta_{00A}.
\end{align}
Letting $A=D=0$ in (\ref{sig2}), we get
\begin{align}\label{ker'}
2\boldsymbol \xi_0\Theta_{0BC} +\boldsymbol \xi_B\Theta_{C00}+\boldsymbol \xi_C\Theta_{B00} =0,
\end{align} i.e.
\begin{align}\label{t01}
\boldsymbol \xi_0\Theta_{0BC}=\boldsymbol \xi_B\Theta_{00 C}+\boldsymbol \xi_C\Theta_{00 B},
\end{align}by (\ref{t00}).

  On the other hand, for   $\Theta\in \ker \sigma''_2(  \xi),$    as in  (\ref{T123}),   we have \begin{equation*}\begin{aligned}
2\left(\sigma_2''(  \xi)\Theta\right)_{EDABC} =&\sum_{(D,B,C)}\left( 2\boldsymbol \xi_{[E}\boldsymbol \xi_{\underline{D}}\Theta_{A]BC} +\boldsymbol \xi_{D}\boldsymbol \xi_{[E}\Theta_{A]BC} +\boldsymbol \xi_{BC}\Theta_{[E\underline{D}A]}\right)
\\=&\left(\boldsymbol \xi_E\boldsymbol \xi_D+ \boldsymbol \xi_{DE}\right)\Theta_{ABC}- \left(\boldsymbol \xi_A\boldsymbol \xi_D+\boldsymbol  \xi_{DA}\right)\Theta_{EBC} + \boldsymbol \xi_{BC}\left(\Theta_{EDA}-\Theta_{ADE}\right) \\&+\left(\boldsymbol \xi_E\boldsymbol \xi_C+\boldsymbol  \xi_{CE}\right)\Theta_{ABD}-\left(\boldsymbol \xi_A\boldsymbol \xi_C+ \boldsymbol  \xi_{CA}\right)\Theta_{EBD} + \boldsymbol  \xi_{BD}\left(\Theta_{ECA} -\Theta_{ACE}\right) \\&+\left(\boldsymbol \xi_E\boldsymbol \xi_B+  \boldsymbol\xi_{BE}\right)\Theta_{ADC} -\left(\boldsymbol \xi_A\boldsymbol \xi_B+\boldsymbol  \xi_{BA}\right)\Theta_{EDC} + \boldsymbol \xi_{DC}\left(\Theta_{EBA} - \Theta_{ABE}\right)=0.
\end{aligned}\end{equation*}
Letting  $D=E=0$ above, we get
  \begin{equation}\begin{aligned}\label{DE0}
0=&3\left|\boldsymbol \xi_0\right|^2\Theta_{ABC} -\left(2\boldsymbol \xi_A\boldsymbol \xi_0  +\boldsymbol \xi_0\boldsymbol \xi_A\right) \Theta_{0BC} +  \boldsymbol\xi_{BC}\left(\Theta_{00A}- \Theta_{A00}\right) \\&+\left(2\boldsymbol \xi_0\boldsymbol \xi_C+\boldsymbol \xi_C\boldsymbol \xi_0\right)\Theta_{AB0} -\left(2\boldsymbol \xi_A\boldsymbol \xi_C+\boldsymbol \xi_C\boldsymbol \xi_A\right)\Theta_{00B} + \boldsymbol \xi_{0B}\left(\Theta_{0CA} - \Theta_{AC0}\right)\\& +\left(2\boldsymbol \xi_0\boldsymbol \xi_B +\boldsymbol \xi_B\boldsymbol \xi_0\right)\Theta_{A0C} -\left(2\boldsymbol \xi_A\boldsymbol \xi_B+\boldsymbol \xi_B\boldsymbol \xi_A\right)\Theta_{00C} +\boldsymbol  \xi_{0C}\left(\Theta_{0BA} - \Theta_{AB0}\right).\end{aligned}\end{equation}
Note that     $\Theta_{A0C}=\Theta_{AC0},$ and \begin{align}\label{0BA}\Theta_{0BA} +\Theta_{AB0}=-\Theta_{BA0},\end{align} by letting $C=0$ in (\ref{TABC}), and  \begin{align}\label{BA0}
\boldsymbol \xi_A\Theta_{0BC}+\boldsymbol \xi_B\Theta_{CA0}+\boldsymbol \xi_C\Theta_{BA0}= -\boldsymbol \xi_0\Theta_{ABC},
\end{align}by letting  $D=0$ in  (\ref{sig2}). We can use (\ref{0BA}) to simplify (\ref{DE0}) to get    { \begin{equation}\begin{aligned} \label{DE1} 0=&3\left|\boldsymbol \xi_0\right|^2\Theta_{ABC} -\left(2\boldsymbol \xi_A\boldsymbol \xi_0+\boldsymbol \xi_0\boldsymbol \xi_A\right)\Theta_{0BC} +3 \boldsymbol \xi_{BC}\Theta_{00 A}\\& -\boldsymbol \xi_0\boldsymbol \xi_C\Theta_{BA0}  +\boldsymbol \xi_C\boldsymbol \xi_0\Theta_{0B A}-\boldsymbol \xi_0\boldsymbol \xi_B\Theta_{CA0} +\boldsymbol \xi_B\boldsymbol \xi_0\Theta_{0C A}\\& -\left(2\boldsymbol \xi_A\boldsymbol \xi_C +\boldsymbol \xi_C\boldsymbol \xi_A\right)\Theta_{00B}  -\left(2\boldsymbol \xi_A\boldsymbol \xi_B +\boldsymbol \xi_B\boldsymbol \xi_A\right)\Theta_{00C}, \end{aligned}\end{equation}}by
 \begin{equation*}\begin{aligned}
 \left(2\boldsymbol \xi_0\boldsymbol \xi_C+\boldsymbol \xi_C\boldsymbol \xi_0\right)\Theta_{AB0}+ \boldsymbol \xi_{0C}\left( \Theta_{0BA} - \Theta_{AB0}\right)=&\boldsymbol \xi_0\boldsymbol \xi_C\left(\Theta_{AB0} +\Theta_{0BA}\right)+\boldsymbol \xi_C\boldsymbol \xi_0\Theta_{0BA}\\=&- \boldsymbol \xi_0\boldsymbol \xi_C\Theta_{BA0} +\boldsymbol \xi_C\boldsymbol \xi_0\Theta_{0BA},
 \end{aligned}\end{equation*} and the same  formula with $B$ and $C$ exchanged. Now substitute (\ref{BA0}) into (\ref{DE1}) further to  get\begin{equation*}\begin{aligned} 0=&4\left|\boldsymbol \xi_0\right|^2\Theta_{ABC}-2\boldsymbol \xi_A\boldsymbol \xi_0\Theta_{0BC} +\boldsymbol \xi_C\boldsymbol \xi_0\Theta_{0B A} +\boldsymbol \xi_B\boldsymbol \xi_0\Theta_{0C A} -\left(2\boldsymbol \xi_A\boldsymbol \xi_C+\boldsymbol \xi_C\boldsymbol \xi_A\right)\Theta_{00B} \\& -2\left(\boldsymbol \xi_A\boldsymbol \xi_B+\boldsymbol \xi_B\boldsymbol \xi_A\right)\Theta_{00C} +3 \boldsymbol \xi_{BC}\Theta_{00 A}\\=&4\left|\boldsymbol \xi_0\right|^2\Theta_{ABC} -4 \boldsymbol \xi_A\boldsymbol \xi_B\Theta_{00C} -4\boldsymbol \xi_A\boldsymbol \xi_C\Theta_{00B}+4  \boldsymbol\xi_{BC}\Theta_{00A},
\end{aligned}\end{equation*}  by using  (\ref{t01}) in the last identity.  The lemma is proved.
\end{proof}
\end{lem}

{\it Proof of Theorem \ref{exact}.} Since $  \xi\in\mathbb R^{kn}\setminus\{\mathbf 0\},$ we can assume $\boldsymbol \xi_0\neq0$ without loss of generality.  For any $\vartheta\in\ker \sigma_0(  \xi),$ we have $\left(\sigma_0(  \xi)\vartheta\right)_A= \boldsymbol \xi_A\vartheta=0,$ for   $A=0,\dots,k-1.$  Then $
0=\boldsymbol \xi_A\boldsymbol \xi_A\vartheta=|\boldsymbol \xi_A|^2\vartheta,
$
by (\ref{nunu}). So $\vartheta=0.$ Hence $\sigma_0$ is injective.

  If   $\eta=\left( \eta_{A}  \right)\in \ker \sigma_1(  \xi),$ we have
\begin{align}\label{bac}
\boldsymbol \xi_A\boldsymbol \xi_B\eta_C +\boldsymbol \xi_A\boldsymbol \xi_C\eta_B -\boldsymbol \xi_C\boldsymbol \xi_B\eta_A-  \boldsymbol \xi_B\boldsymbol \xi_C\eta_A=0,
\end{align}by definition  (\ref{D0}). Letting $C= B=0$ in (\ref{bac}), we have
\begin{align}\label{AC}
\boldsymbol \xi_A\boldsymbol \xi_0\eta_0=\left|\boldsymbol \xi_0\right|^2\eta_A.
\end{align}
Set $\vartheta:=\frac{\boldsymbol \xi_{0}\eta_0}{|\boldsymbol \xi_0|^2}\in\mathcal V_0.$ Then \begin{align}\label{xixi}\left(\sigma_0(  \xi)\vartheta\right)_A =\boldsymbol \xi_A\frac{\boldsymbol \xi_{0}\eta_0}{|\boldsymbol \xi_0|^2} =\eta_A,\end{align}
for   $A=0,\dots,k-1,$ i.e. $\ker \sigma_1(  \xi)\subset{\rm Im}\ \sigma_0(  \xi).$  It follows from  $\mathcal D_{1}\circ\mathcal D_0=0$  that $\sigma_{1}(  \xi)\circ\sigma_0(  \xi)=0.$   So we have  ${\rm Im}\  \sigma_0(  \xi)=\ker\sigma_{1}(  \xi).$

By Theorem \ref{pell},  $\mathcal D_2'\circ\mathcal D_1=0,\ \mathcal D_2''\circ\mathcal D_1=0.$ It follows   that $\sigma'_2( \xi)\circ\sigma_1(  \xi)=0,\ \sigma''_2(  \xi)\circ\sigma_1(  \xi)=0,$ for any  $  \xi\in\mathbb R^{kn}\setminus\{\mathbf 0\},$ respectively. Thus $\ker \sigma'_2(  \xi)\cap\ker \sigma''_2(  \xi) \supset{\rm Im}\ \sigma_1(  \xi).$ Conversely, for  $\Theta\in\ker \sigma'_2(  \xi)\cap\ker \sigma''_2(  \xi),$ take $\eta\in\mathcal V_1$ with   $$\eta_A:=\frac{2\Theta_{00A}}{\left|\boldsymbol \xi_0\right|^2},$$ $A=0,\dots,k-1.$ Then we have \begin{equation*}\begin{aligned}
 \left(\sigma_1(  \xi)\eta\right)_{ABC}=&\frac{1}{2} \left(\boldsymbol \xi_A\boldsymbol \xi_B\eta_C +\boldsymbol \xi_A\boldsymbol \xi_C\eta_B-\boldsymbol  \xi_{BC}\eta_A\right) \\=& \frac{1}{\left|\boldsymbol \xi_0\right|^2}\left(\boldsymbol \xi_A\boldsymbol \xi_B\Theta_{00C} +\boldsymbol \xi_A\boldsymbol \xi_C\Theta_{00B}- \boldsymbol \xi_{BC}\Theta_{00A}\right) =\Theta_{ABC},
\end{aligned}\end{equation*} by (\ref{tabc}) in Lemma \ref{l51}.
Thus we have $\ker \sigma'_2(  \xi)\cap\ker \sigma''_2(  \xi)\subset{\rm Im}\ \sigma_1(  \xi).$
 The theorem is proved.
\qed

\subsection{The uniform ellipticity of the associated Hodge-Laplacian operators}
The natural Hodge Laplacian associated to the differential  complex (\ref{co3}) should be
\begin{align*}
\widetilde\Box_j:=\mathcal D_{j-1}\mathcal D_{j-1}^*+\mathcal D_j^*\mathcal D_j,
\end{align*}
for $j=1,2.$ But   $\mathcal D_1$ and ${\mathcal D''_2}$ are  differential operators of second order, while $\mathcal D_0,{\mathcal D}'_2$ are of first order. Thus  $\widetilde\Box_1$ and $\widetilde\Box_2$ have degenerate principal symbols and  are not uniformly elliptic. So it is better to  consider the associated Hodge  Laplacians of fourth order as \begin{equation}\begin{aligned}\label{hodge}
\Box_0:=&\left(\mathcal D_0^*\mathcal D_0\right)^2,\\\Box_1:=&\left(\mathcal D_0\mathcal D_0^*\right)^2+\mathcal D_1^*\mathcal D_1,\\\Box_2:=&\mathcal D_1\mathcal D_1^*+\left({\mathcal{D}_2'}^*{\mathcal{D}_2'}\right)^2 +{\mathcal{D}_2''}^*{\mathcal{D}_2''}.
\end{aligned}\end{equation}  It is easy to see that  $\Box_0=\Delta^2$ by Corollary \ref{p23}.
\begin{prop}\label{boxj}
Let $L_j(  \xi):=\sigma\left(\Box_j\right)(  \xi)$ and $r_j={\rm dim} \left(\mathcal V_j\right).$ Then
\begin{equation*}\begin{aligned}
L_0(  \xi)=&\left(\sigma_0(  \xi)^*  \sigma_0(  \xi)\right)^2,\\L_1(  \xi)=&\left(\sigma_0(  \xi)  \sigma_0(  \xi)^*\right)^2+\sigma_1(  \xi)^*  \sigma_1(  \xi),\\L_2(  \xi)=&\sigma_1(  \xi)  \sigma_1(  \xi)^*+\left(\sigma'_2(  \xi)^*  \sigma'_2(  \xi)\right)^2+\sigma_2''(  \xi)^*  \sigma_2''(  \xi),
\end{aligned}\end{equation*}
are all automorphisms of $\mathbb C^{r_j}$  respectively, for any fixed dual variable $  \xi\in\mathbb R^{kn}\setminus\{\mathbf 0\},$   and are all homogeneous of degree $4$ in $  \xi.$
\end{prop}
\begin{proof} The proof is similar to \cite{wang2008} \cite{Wa10} for
  $k$-Cauchy-Fueter complexes. We sketch it for $L_1(  \xi).$  By definition,  $L_1(  \xi)$ is  homogeneous of degree $4$ in $  \xi.$
Note that ${\rm range}\ \sigma_1(  \xi)\perp\ker  \sigma_1(  \xi)^*$, and
\begin{equation}\label{eq:L-decomposition}
 \mathbb C^{r_1}={\rm range}\ \sigma_1(  \xi)^*\oplus\ker  \sigma_1(  \xi)
\end{equation}
  are orthogonal decompositions  by the ellipticity (\ref{simco}) and simple linear algebra. Therefore
  $
{\rm range}\ \sigma_1(  \xi)^* \sigma_1(  \xi)= {\rm range}\ \sigma_1(  \xi)^*.
$
Similarly
${\rm range}\ \sigma_0(  \xi)^*\oplus\ker  \sigma_0(  \xi)=\mathbb C^{r_0}$ is a orthogonal decomposition, and so
\begin{equation} 
 \label{eq:L-decomposition2}{\rm range}\ \sigma_0(  \xi)  \sigma_0(  \xi)^*={\rm range}\ \sigma_0(  \xi)=\ker \sigma_1(  \xi).
\end{equation}
 Substitute them into the decomposition \eqref{eq:L-decomposition} to get
\begin{equation*}\begin{aligned}
\mathbb C^{r_1}=&{\rm range}\ \sigma_0(  \xi)  \sigma_0(  \xi)^*\oplus {\rm range}\ \sigma_1(  \xi)^* \sigma_1(  \xi),
\end{aligned}\end{equation*}
for $  \xi\in\mathbb R^{kn}\setminus\{\mathbf 0\}.$   We also have ${\rm range}\left(\sigma_0(  \xi)  \sigma_0(  \xi)^*\right)^2={\rm range}\ \sigma_0(  \xi) \sigma_0(  \xi)^*  $ by using \eqref{eq:L-decomposition2}, $\sigma_0(  \xi)^* \sigma_1(  \xi)^* =0$ and the  decomposition  \eqref{eq:L-decomposition}. Therefore,
  $$\mathbb C^{r_1}={\rm range}\left( \sigma_0(  \xi)  \sigma_0(  \xi)^*\right)^2\oplus {\rm range}\ \sigma_1(  \xi)^* \sigma_1(  \xi),$$
 Then the linear transformation $L_1(  \xi)$ is automorphisms of $\mathbb C^{r_1}$ for any fixed  $  \xi\in\mathbb R^{kn}\setminus\{\mathbf 0\}.$

Similarly, $L_0(  \xi)$ and $L_2(  \xi)$ are also automorphisms of $\mathbb C^{r_0}$ and  $\mathbb C^{r_2},$ respectively.
\end{proof}
Proposition \ref{boxj} implies that the associated Laplacian operators in (\ref{hodge}) are all uniformly elliptic differential operators of fourth order with constant coefficients, i.e. there exists a constant $C > 0$ such that
\begin{align*}
C^{-1}|  \xi|^4I_{r_j}\leq L_j(  \xi)\leq C |  \xi|^4I_{r_j}.
\end{align*}

Let $G_j(  \xi)=L_j(  \xi)^{-1}$ for $  \xi\in\mathbb R^{kn}\setminus\{\mathbf 0\}.$ It is a smooth  homogeneous functions of degree $4-kn.$ Let $\mathbf G_j$ be the convolution operator with kernel $G_j.$ The following result in harmonic analysis can be applied to $\mathbf G_j$.
\begin{prop}{\rm(\cite[Proposition 2.4.7]{Grafakos})}\label{PV} Let $K\in C^\infty\left(\mathbb R^{kn}\setminus\{\mathbf 0\}\right)$   be a homogeneous function of degree $l-kn,$ and  let $\mathbf K$ be the operator defined by $\mathbf K\phi = \phi *K.$
Then, for $\phi\in C_c^\infty\left(\mathbb R^{kn}\right)$,  $A_1,\dots,A_l=0,\dots,k-1,j_1,\dots, j_l=1,\dots, n,$
\begin{align}\label{3.2}
\partial_{A_1j_1}\dots\partial_{A_lj_l}(\mathbf K\phi)=P.V.\left(\phi*\partial_{A_1j_1} \dots\partial_{A_lj_l}K\right) +a_{A_1j_1\dots A_lj_l}\phi,
\end{align}
for some constant  $a_{A_1j_1\dots A_lj_l}.$ Each term in {\rm (\ref{3.2})} is $C^\infty$ and the identity holds as $C^\infty$ functions.

Moreover, $\partial_{A_1j_1}\dots\partial_{A_lj_l}  K$  is a Calderon-Zygmund kernel on $\mathbb R^{kn}.$ The singular integral operator $f\rightarrow P.V.$ $ \left(f *\partial_{A_1j_1}\dots\partial_{A_lj_l}K\right)$ is bounded
on $L^p$ for $1 < p <\infty.$ {\rm (\ref{3.2})} holds as $L^p$ functions.
\end{prop}

\begin{prop}\label{GG}
The associated Hodge-Laplacian $\Box_j$ has the inverse $\mathbf G_j$ in $L^2\left(\mathbb R^{kn},\mathcal V_j\right), j=0,1,2 $.  $\mathbf G_j$ can be extended to a bounded linear operator from $L^p\left(\mathbb R^{kn},\mathcal V_j\right)$ to $W^{4,p}\left(\mathbb R^{kn},\mathcal V_j\right),$ for $1 < p <\infty.$ For  a $\mathcal V_j$-valued function $f$   in $C_c^l\left(\mathbb R^{kn},\mathcal V_j\right)$ and any non-negative integer $l,$ we have $\mathbf G_jf\in C^{l+3}\left(\mathbb R^{kn},\mathcal V_j\right).$
\end{prop}
\begin{proof}
For   $f\in \Gamma_c \left(\mathbb R^{kn},\mathcal V_j\right),$  \begin{align}\label{gG}
\mathbf G_jf=\int_{\mathbb R^{kn}}G_j(\mathbf x-\mathbf y)f(\mathbf y)\hbox{d}V(\mathbf y),
\end{align}for $j=0,1,2.$   Then for $f\in \Gamma_c \left(\mathbb R^{kn},\mathcal V_j\right),$ \begin{align}\label{Gj}
\mathbf G_j\Box_jf=\Box_j\mathbf G_jf=f,
\end{align}holds as  $L^2 $ functions,  since  their Fourier transformations are the same as $L^2$ functions. Since $L_j( \xi)$ are symmetric matrices by Proposition \ref{boxj}, we see that $G_j( \xi)$ are also symmetric. It follows from Plancherel theorem that   $\mathbf G_j, j = 0, 1, 2,$ is formally self-adjoint in the following sense:
\begin{align*}
\left(\mathbf G_j\phi,\psi\right)_{L^2\left(\mathbb R^{kn},\mathcal V_j\right)}=\left(\phi,\mathbf G_j\psi\right)_{L^2\left(\mathbb R^{kn},\mathcal V_j\right)},
\end{align*}
for any $\phi,\psi\in \Gamma_c \left(\mathbb R^{kn},\mathcal V_j\right).$

On the other hand,  $G_j$ satisfies the following decay estimates
\begin{align}\label{partial}
\left|\partial_{{A_1j_1}} \dots\partial_{{A_mj_m}} G_j(\mathbf x)\right|\leq \frac{C_{A_1j_1\dots A_mj_m}}{|\mathbf x|^{kn-4+m}},
\end{align}
for some constant $C_{A_1j_1\dots A_mj_m}>0$ depending on  $A_1,\dots,A_m=0,\dots,k-1,$ $j_1,\dots,$ $j_m=1,\dots,n.$
By   Proposition \ref{PV}, the convolution with a homogeneous function of degree $-kn+l$ can be extended to a bounded linear operator from $L^p\left(\mathbb R^{kn}\right)$ to $W^{l,p}\left(\mathbb R^{kn}\right).$  $\mathbf G_j$ is bounded from $L^2\left(\mathbb R^{kn},\mathcal V_j\right)$ to $W^{4,2}\left(\mathbb R^{kn},\mathcal V_j\right)$. So $\mathbf G_j$ can be extended to a bounded operator from $W^{-4,2}\left(\mathbb R^{kn},\mathcal V_j\right)$ to $L^2\left(\mathbb R^{kn},\mathcal V_j\right)$ by the duality argument. In particular, it is bounded from $W^{-2,2}\left(\mathbb R^{kn},\mathcal V_j\right)$ to $L^2\left(\mathbb R^{kn},\mathcal V_j\right).$  Thus (\ref{Gj}) holds as bounded linear operator from $L^2$ to $L^2,$ i.e.  $\mathbf G_j$ is the inverse operator of  Laplacian $\Box_j$ in $L^2.$

When $f\in C_c^{l}\left(\mathbb R^{kn},\mathcal V_j\right),$    $\mathbf G_jf\in C^{l+3}\left(\mathbb R^{kn},\mathcal V_j\right)$  by differentiation (\ref{gG}). The proposition is proved.
\end{proof}
\subsection{Solutions of non-homogeneous several Dirac equations}
To show the Hartogs-Bochner extension, we need to solve non-homogeneous several Dirac equations
\begin{align}\label{duf}
\mathcal D_0u=f,
\end{align}
 under the compatibility  condition
\begin{align}\label{comp}
\mathcal D_{1}f=0.
\end{align}

\begin{thm}\label{t31}
Suppose that $f\in L^2\left(\mathbb R^{kn},\mathcal V_1\right)$ satisfies the compatibility condition {\rm(\ref{comp})} in sense of distributions. Then there exists   $u\in W^{1,2}\left(\mathbb R^{kn},\mathcal V_0\right)$ satisfying the non-homogeneous   equation {\rm(\ref{duf})}. Furthermore, if $f\in C_c\left(\mathbb R^{kn},\mathcal V_1\right)$ with $\mathcal D_1f=0$ in the sense of distributions,  then there exists   $u\in C_c\left(\mathbb R^{kn},\mathcal V_0\right)\cap W^{1,2}\left(\mathbb R^{kn},\mathcal V_0\right)$ satisfying {\rm(\ref{duf})} and vansihing on the unbounded connected component of $\mathbb R^{kn}\setminus {\rm supp}f.$
\end{thm}
 \begin{proof}
  Recall that  $\mathbf G_1$ is the inverse operator of $\Box_1$ on $L^2\left(\mathbb R^{kn},\mathcal V_1\right)$ by Proposition \ref{GG}. We assume $f\in \Gamma_c \left(\mathbb R^{kn},\mathcal V_1\right)$ first. Then
\begin{align*}
u:=\mathcal D_0^*\mathcal D_0\mathcal D_0^*\mathbf G_1f\in C^\infty\left(\mathbb R^{2n},\mathcal V_0\right),
\end{align*}  by Proposition \ref{GG}. Note that
\begin{equation}\begin{aligned}\label{d1g1}
\mathcal D_1\Box_1f=&\mathcal D_1\left(\mathcal D_1^*\mathcal D_1+\left(\mathcal D_0\mathcal D_0^*\right)^2\right)f=\mathcal D_1\mathcal D_1^*\mathcal D_1f\\=&\left(\mathcal D_1\mathcal D_1^*+\left({\mathcal D'_2}^*{\mathcal D'_2}\right)^2+{\mathcal{D}_2''}^*{\mathcal{D}_2''}\right)\mathcal D_1f=\Box_2\mathcal D_1f,
\end{aligned}\end{equation}
   by $\mathcal D_1\mathcal D_0=0,{\mathcal D'_2}\mathcal D_1=0,{\mathcal D''_2}\mathcal D_1=0.$
Now  for $f\in \Gamma_c \left(\mathbb R^{kn},\mathcal V_1\right),$ we have
\begin{equation}\begin{aligned}\label{b2}
\Box_2 \left(\mathbf G_2\mathcal D_1f-\mathcal D_1\mathbf G_1f\right)=\mathcal D_1f-\mathcal D_1\Box_1\mathbf G_1f=0,
\end{aligned}\end{equation}
by using (\ref{d1g1}). For $f\in \Gamma_c  \left(\mathbb R^{kn},\mathcal V_1\right),$ it is direct to see that    $\left(\mathbf G_2\mathcal D_1f-\mathcal D_1\mathbf G_1f\right)\widehat{} \in L^2$, and $L_2(  \xi)\left(\mathbf G_2\mathcal D_1f-\mathcal D_1\mathbf G_1f\right)\widehat{} (  \xi)=0$ for $  \xi\in\mathbb R^{kn} $ by (\ref{b2}). Thus $\left(\mathbf G_2\mathcal D_1f-\mathcal D_1\mathbf G_1f\right)\widehat{} =0$ in $L^2,$ and so  by Plancherel theorem   $ \mathbf G_2\mathcal D_1f=\mathcal D_1\mathbf G_1f $ in $L^2.$ Moreover, $ \mathbf G_2\mathcal D_1f=\mathcal D_1\mathbf G_1f$ pointwisely, since they are both smooth  by Proposition \ref{GG}.  Then \begin{align}\label{GD}\mathbf G_2\mathcal D_1=\mathcal D_1\mathbf G_1,\end{align}as operators      from $\Gamma_c \left(\mathbb R^{kn},\mathcal V_1\right)$ to $W^{2,2}\left(\mathbb R^{kn},\mathcal V_2\right)$ by Proposition \ref{GG}.
 Thus \begin{align}\label{DGf}\mathcal D_1\mathbf G_1f=\mathbf G_2\mathcal D_1f=0,\end{align} and so \begin{align}\label{DDDDG}
\mathcal D_0u=\mathcal D_0\mathcal D_0^*\mathcal D_0\mathcal D_0^*\mathbf G_1f=\left(\left(\mathcal D_0\mathcal D_0^*\right)^2+\mathcal D_1^*\mathcal D_1\right)\mathbf G_1f=f,
\end{align}
i.e. $u=\mathcal D_0^*\mathcal D_0\mathcal D_0^*\mathbf G_1f$ satisfies (\ref{duf}) for $f\in \Gamma_c \left(\mathbb R^{2n},\mathcal V_1\right).$ Since $\mathbf G_j$  is bounded from $W^{-2,2} (\mathbb R^{kn},$ $\mathcal V_j )$ to $L^2\left(\mathbb R^{kn},\mathcal V_j\right) $ by the proof of Proposition \ref{GG},   the identity (\ref{GD}) holds as   bounded linear operators from  $L^2\left(\mathbb R^{kn},\mathcal V_1\right)$ to $L^2\left(\mathbb R^{kn},\mathcal V_2\right)$ and so the identity (\ref{DGf}) holds for $f\in L^2\left(\mathbb R^{kn},\mathcal V_1\right).$ Thus (\ref{DDDDG}) holds as $L^2$ functions for  $f\in L^2\left(\mathbb R^{kn},\mathcal V_1\right)$ satisfying $\mathcal D_1f = 0$ in the sense of distributions. Thus    $u =\mathcal D_0^*\mathcal D_0\mathcal D_0^*\mathbf G_1f$ satisfies the  equation (\ref{duf}).

Suppose that $f$ is supported in a bounded $\Omega\Subset\mathbb R^{2n}.$   By $\mathcal D_0u= f$, we  see that $u$ is monogenic on $\mathbb R^{kn}\setminus \Omega.$ Since the integral kernel $K(\mathbf x)$ of $\mathcal D_0^*\mathcal D_0\mathcal D_0^*\mathbf G_1$ decays as $|\mathbf x|^{-kn+1}$ for large $|\mathbf x|$ by decay estimate (\ref{partial}) and $f$ is compactly supported, we see that
\begin{align}\label{k}
\left|u(\mathbf x)\right|=\left|\int_{\mathbb R^{kn}}K\left(\mathbf x-{\mathbf y}\right)f\left({\mathbf y}\right)\hbox{d}V\left({\mathbf y}\right)\right|\leq\frac{C}{\left(1+|\mathbf x|\right)^{kn-1}},
\end{align}
for some constant $C>0.$
So  $\lim_{|\mathbf x|\rightarrow\infty}u(\mathbf x)=0.$ Since
$\Omega$ is bounded, its projection to $\mathbb R_{\mathbf x_0}^{ n} $ is also bounded and so is contained in a ball centered at origin with   radius, say $M$.
For $\mathbf x=\left(\mathbf x_0,\mathbf x'\right)$ with $\mathbf x'=\left(\mathbf x_1,\dots,\mathbf x_{k-1}\right),$
 if $|\mathbf x_0|>M$, then
 $$\left(\left\{\mathbf x_0\right\}\times\mathbb R_{\mathbf x'}^{(k-1)n}\right)\cap\overline \Omega=\emptyset.$$
  Thus, $u\left(\mathbf x_0,\mathbf x'\right)$ for $\left|\mathbf x_0\right|>M$ is   monogenic  in variable $\mathbf x'\in \mathbb{R}^{(k-1)n}$, and   vanishes at infinity. Since
\begin{align}\label{har}
\Delta'u\left(\mathbf x_0,\mathbf x'\right)=0,
\end{align}by Corollary \ref{p34}, where $\Delta'= - \sum_{A=1}^{k-1}\sum_{j=1}^n\partial_{Aj}^2,$
each component of $u\left(\mathbf x_0, \cdot\right)$ is a biharmonic function on $\mathbb R_{\mathbf x'}^n$ vanishing at infinity,  and so it is bounded. Hence $u\left(\mathbf x_0, \cdot\right)\equiv0$ for $|\mathbf x_0|>M$   by Liouville-type theorem. So $u\equiv0$ on the unbounded connected component of $\mathbb R^{kn}\setminus \Omega$ by the identity theorem for real analytic functions, since  $u$ is real analytic on $\mathbb R^{kn}\setminus \Omega$ by Corollary \ref{p34}. The continuity of $u$ follows from the integral representation  formula as in (\ref{k}). The theorem is proved.
\end{proof}

 Now we can prove the  Hartogs-Bochner extension theorem  for tangentially  monogenic functions.

{\it Proof of Theorem \ref{HBE}.}
By Proposition \ref{phat}, we can extend  $f$ to a smooth function $\tilde f$ on $\overline \Omega$ such that $\left.\tilde f\right|_{\partial \Omega}=f$ and $\mathcal D_0\tilde f$ is flat on $\partial \Omega.$ Then we   extend $\mathcal D_0\tilde f$ by $0$ outside of $\overline \Omega$ to get a $\mathcal D_1$-closed element $F\in C_c^2\left(\mathbb R^{kn},\mathcal V_1\right)$ supported in $\overline \Omega.$ Then by Theorem \ref{t31}, there exists $H\in W^{1,2}\left(\mathbb R^{kn},\mathcal V_0\right)$ vanishing  on the connected open set $\mathbb R^{kn}\setminus\overline \Omega$   such that $F=\mathcal D_0H.$ Then $ f-H$ monogenic on $\Omega$ and  gives us the required extension.
\qed
\appendix\section{The construction of Weyl modules}\label{app}

Young diagrams can be used to describe projection operators for the regular representation, which will then give the irreducible representations of ${\rm GL}(k).$ More generally,   a \emph{tableau} is  Young diagram with   numbering
of the boxes by the integers $1, \dots ,k.$ Given a tableau  define two subgroups of the symmetric group
\begin{equation*}
P=P_\lambda=\left\{g\in{\rm GL}(k): g\ {\rm preserves\ each\ row}\right\}
\end{equation*}
and
\begin{equation*}
Q=Q_\lambda=\left\{g\in{\rm GL}(k): g\ {\rm preserves\ each\ column}\right\}.
\end{equation*}
In the group algebra $\mathbb C{\rm GL}(k),$ we introduce two elements corresponding to these subgroups and set
\begin{equation*}
a_\lambda:=\sum_{g\in P}e_g,\quad   b_\lambda:=\sum_{g\in Q}{\rm sgn}(g)e_g,,
\end{equation*} where $\left\{e_g\right\}$ is the basis of $\mathbb C{\rm GL}(k).$
The {\it Young symmetrizer} \begin{align*}
\mathfrak c_\lambda=a_\lambda\cdot b_\lambda
\end{align*}
 in  $\mathbb C{\rm GL}(k)$   defines an action  $ \otimes^k\mathbb C^k\rightarrow\otimes^k\mathbb C^k $ by
 \begin{equation}\begin{aligned}\label{mfc}
 \left(\omega^{i_k\dots v_{i_1}} \right) \mathfrak c_\lambda=\sum_{\sigma=\sigma_1\sigma_2\in PQ}{\rm sgn}\left(\sigma_2\right)\cdot \omega^{i_{\sigma(k)} \dots  i_{\sigma(1)}}.
\end{aligned}\end{equation}
The image of $\mathfrak c_\lambda$ on ${\otimes^k}\mathbb C^k$ is called the  \emph{Weyl module}, which is isomorphic to an irreducible
 representation of ${\rm GL}(k)$     labelled by $\lambda$ (cf. \cite[Section 6.1]{FH}).
\begin{prop}{\rm (cf. \cite[Theorem 4.3, Lemma 4.26]{FH})} $\frac{\mathfrak c_\lambda}{n_{\lambda}}$ is idempotent, i.e. $\mathfrak c_\lambda \mathfrak c_\lambda=n_\lambda \mathfrak c_\lambda,$ where $n_\lambda=\frac{k!}{{\rm dim} V_\lambda}$ with $V_\lambda$ be   an irreducible representation ${\rm GL}(k).$
\end{prop}

For Young diagram
\begin{equation*}
\begin{ytableau}
       1 & 2  \\
  3
\end{ytableau},
\end{equation*}
denote   \begin{align}\label{C21}
\mathfrak c_{21}:=\frac13\{1+(12)\}\{1-(13)\} =\frac13\left[1+(12)-(13)-(132)\right].
\end{align}
Here we add the factor $\frac13$ so that it becomes a projection operator, compared to the original definition in (\ref{mfc}). For $\omega^{321}  \in\otimes^3\mathbb C^k,$ we have
\begin{align}
\left(\omega^{321} \right)\mathfrak c_{21}=\omega^{321} +\omega^{312}-\omega^{123}  -\omega^{132},
\end{align}by (\ref{C21}). Here the element $\mathfrak c_{21}$ of the group algebra $\mathbb C{\rm GL}(k)$ acts on $\omega^{321}$ from right. If   label the Young diagram by letter $A,B,C$ as  $\begin{ytableau}C & B  \\  A\end{ytableau},$ for an element $h_{ABC}\omega^{ABC}\in\otimes^3\mathbb C^k,$ we have \begin{equation*}\begin{aligned}
 \left(h_{ABC}\omega^{ABC} \right)\mathfrak c_{21}=&\frac13 h_{ABC} \left(\omega^{ABC} +\omega^{ACB} -\omega^{CBA} -\omega^{CAB} \right) \\=&\frac23\left(h_{[A\underline{B}C]}+h_{[A\underline{C}B]} \right) \omega^{ABC},
\end{aligned}\end{equation*} by relabeling indices.

 For Young diagram
$
\begin{ytableau}
       1 & 2  \\
  3&4
\end{ytableau},
$
denote  \begin{align}\label{C22}
\mathfrak c_{22}:=\frac{1}{12}\{1+(12)\}\{1+(34)\}\{1-(13)\}\{1-(24)\}.
\end{align}
For $\omega^{4321}\in \otimes^4 \mathbb C^k,$ we have
 \begin{equation*}\begin{aligned}
\left(\omega^{4321} \right)\mathfrak c_{22}=&\frac{1}{12}\left(\omega^{4321} +\omega^{4312} +\omega^{3421} +\omega^{3412}  -\omega^{4123} -\omega^{4132} -\omega^{1423} -\omega^{1432}\right. \\&\left.\qquad-
\omega^{2341} -\omega^{2314} -\omega^{3241} -\omega^{3214}  +
\omega^{2143} +\omega^{2134} +\omega^{1243} +\omega^{1234} \right),
\end{aligned}\end{equation*} by   (\ref{C22}) of $\mathfrak c_{22}.$ If   label the Young diagram by letter $A,B,C,D$ as $\begin{ytableau}      C & B  \\  A&D\end{ytableau}$,  for an  element $h_{DABC}\omega^{DABC}\in\otimes^4\mathbb C^k,$ we have  {\small\begin{equation*}\begin{aligned}
& \left(h_{DABC}\omega^{DABC} \right)\mathfrak c_{22}\\=&\frac{1}{12}h_{DABC}
\left(\omega^{DABC}+\omega^{DACB}+\omega^{ADBC}+\omega^{ADCB}  - \omega^{DCBA}-\omega^{DCAB}-\omega^{CDBA}-\omega^{CDAB}\right. \\&\qquad\qquad\quad-\left.\omega^{BADC}-\omega^{BACD} -\omega^{ABDC}-\omega^{ABCD} +\omega^{BCDA}+\omega^{BCAD}+\omega^{CBDA}+\omega^{CBAD}\right)
\\=&\frac16\sum_{(A,D),(B,C)}\left(h_{D[A\underline{B}C]}+ h_{B[C\underline{A}D]}\right).
\end{aligned}\end{equation*}}

 For Young diagram
$
\begin{ytableau}
       1 & 2 & 4 \\
  3\\5
\end{ytableau},
$
denote  {\small\begin{align}\label{C311}
\mathfrak c_{311}:=\frac{1}{20}\{1+(12) +(14)+(24)+(124)+(142)\}\{1-(13)-(15)-(35)+(135)+(153)\}.
\end{align}}    For $\omega^{54321} \in\otimes^5\mathbb C^k,$ we have
 \begin{equation*}\begin{aligned}
\left(\omega^{54321}\right)\mathfrak c_{311}=&\frac{1}{20}\left(\omega^{54321} +\omega^{54312} +\omega^{51324} +\omega^{52341} +\omega^{51342} +\omega^{52314} \right.
\\&\qquad-\omega^{54123} -\omega^{54132} -\omega^{53124}   -\omega^{52143} -\omega^{53142} -\omega^{52134}
\\&\qquad-\omega^{14325} -\omega^{14352} -\omega^{15324}   -\omega^{12345} -\omega^{15342} -\omega^{12354}
\\&\qquad-\omega^{34521} -\omega^{34512} -\omega^{31524}  -\omega^{32541} -\omega^{31542} -\omega^{32514}
\\&\qquad+\omega^{14523} +\omega^{14532} +\omega^{13524}   +\omega^{12543} +\omega^{13542} +\omega^{12534}
\\&\qquad\left.+\omega^{34125} +\omega^{34152} +\omega^{35124}+\omega^{32145} +\omega^{35142} +\omega^{32154}
\right).
\end{aligned}\end{equation*}

If   label the Young diagram by letter $A,B,C,D,E$ as
$\begin{ytableau}  C & B & D \\ A\\E\end{ytableau},$ then for an element\\
 $h_{EDABC}$ $\omega^{EDABC}\in\otimes^5\mathbb C^k,$ we have     {\small \begin{equation*}\begin{aligned}
 &\left(h_{EDABC}\omega^{EDABC}\right)\mathfrak c_{311}\\
=&h_{EDABC}\left(\omega^{EDABC} +\omega^{EDACB} +\omega^{ECABD} +\omega^{EBADC} +\omega^{ECADB} +\omega^{EBACD}\right.
\\&\qquad\qquad-\omega^{EDCBA} -\omega^{EDCAB} -\omega^{EACBD}  -\omega^{EBCDA} -\omega^{EACDB} -\omega^{EBCAD}
\\&\qquad\qquad-\omega^{CDABE} -\omega^{CDAEB} -\omega^{CEABD} -\omega^{CBADE} -\omega^{CEADB} -\omega^{CBAED}
\\&\qquad\qquad-\omega^{ADEBC} -\omega^{ADECB} -\omega^{ACEBD}  -\omega^{ABEDC} -\omega^{ACEDB} -\omega^{ABECD}
\\&\qquad\qquad+\omega^{CDEBA} +\omega^{CDEAB} -\omega^{CAEBD} +\omega^{CBEDA} +\omega^{CAEDB} +\omega^{CBEAD}
\\&\qquad\qquad\left.+\omega^{ADCBE} +\omega^{ADCEB} +\omega^{AECBD} +\omega^{ABCDE} +\omega^{AECDB} +\omega^{ABCED}
\right)
\\=&\frac{3}{10}\sum_{(D,B,C)} h_{[E\underline{D}A\underline{B}C]}\omega^{EDABC}.
\end{aligned}\end{equation*}}

\end{document}